\documentclass[11pt]{amsart}

\RequirePackage{epigamath}

\usepackage[english]{babel}

\numberwithin{equation}{section}

\usepackage{enumitem}
\usepackage[all]{xy}
\usepackage{mathrsfs}
\usepackage{tikz}
\usepackage[utf8]{inputenc}

\newtheorem{THM}{Theorem}

\newtheorem{thm}{Theorem}[section]
\newtheorem{step}{Step}
\newtheorem{lem}[thm]{Lemma}
\newtheorem{corol}[thm]{Corollary}
\newtheorem*{corol*}{Corollary}
\newtheorem{prop}[thm]{Proposition}
\newtheorem{conj}{Conjecture}
\newtheorem{prob}{Problem}
\newtheorem*{thm*}{Theorem}
\newtheorem*{cnj*}{Conjecture}

\theoremstyle{definition}
\newtheorem{rmk}[thm]{Remark}
\newtheorem{eg}[thm]{Example}
\newtheorem{dfn}[thm]{Definition}

\newtheorem*{akn}{Acknowledgements}
\newtheorem*{conj*}{Conjecture}

\newcommand{\cA}{\mathcal{A}}
\newcommand{\cB}{\mathcal{B}}

\newcommand{\cD}{\mathcal{D}}
\newcommand{\cH}{\mathcal{H}}
\newcommand{\cE}{\mathcal{E}}

\newcommand{\cL}{\mathcal{L}}
\newcommand{\cR}{\mathcal{R}}
\newcommand{\cF}{\mathcal{F}}
\newcommand{\cM}{\mathcal{M}}
\newcommand{\cG}{\mathcal{G}}
\newcommand{\cV}{\mathcal{V}}
\newcommand{\cT}{\mathcal{T}}
\newcommand{\cN}{\mathcal{N}}
\newcommand{\cS}{\mathcal{S}}

\newcommand{\cI}{\mathcal{I}}

\newcommand{\cU}{\mathcal{U}}
\newcommand{\cO}{\mathcal{O}}
\newcommand{\cHom}{\mathcal{H}om}
\newcommand{\cExt}{\mathcal{E}xt}

\DeclareMathOperator{\cha}{\mathrm{char}}

\DeclareMathOperator{\cok}{coker}
\DeclareMathOperator{\gr}{gr}
\DeclareMathOperator{\pp}{\mathrm{p}}
\DeclareMathOperator{\id}{\mathrm{id}}

\DeclareMathOperator{\rk}{rk}
\DeclareMathOperator{\Ext}{Ext}
\DeclareMathOperator{\Tor}{Tor}
\DeclareMathOperator{\Hom}{Hom}
\DeclareMathOperator{\End}{End}
\DeclareMathOperator{\im}{Im}
\DeclareMathOperator{\Hilb}{Hilb}

\DeclareMathOperator{\HH}{H}

\DeclareMathOperator{\grade}{grade}

\newcommand{\ZZ}{\mathbb Z}

\newcommand{\QQ}{\mathbb Q}

\newcommand{\NN}{\mathbb N}
\newcommand{\PP}{\mathbb P}

\newcommand{\sx}{{\sf x}}

\newcommand{\bk}{{\boldsymbol k}}
\newcommand{\bM}{\boldsymbol M}

\newcommand{\Proj}{\mathrm{Proj}}
\newcommand{\Mod}{\mathbf{Mod}}
\newcommand{\FMod}{\mathbf{FMod}}
\newcommand{\MCM}{\mathbf{MCM}}
\newcommand{\SMCM}{\underline{\mathbf{MCM}}}
\newcommand{\SHom}{\underline{\mathrm{Hom}}}
\newcommand{\SOm}{\underline{\boldsymbol{\Omega}}}
\newcommand{\Om}{\boldsymbol{\Omega}}
\newcommand{\Rep}{\mathbf{Rep}}
\newcommand{\Ulr}{\mathbf{Ulr}}
\newcommand{\Coh}{\mathbf{Coh}}
\newcommand{\ACM}{\mathbf{ACM}}
\newcommand{\bD}{\mathbf{D}^b}
\newcommand{\bR}{\mathbf{R}}

\DeclareMathOperator{\ts}{\otimes}
\newcommand{\mono}{\hookrightarrow}
\newcommand{\epi}{\twoheadrightarrow}
\newcommand{\la}{\leftarrow}
\newcommand{\Sing}{\mathrm{Sing}}
\newcommand{\lx}{\xleftarrow}
\newcommand{\rx}{\xrightarrow}
\newcommand{\APic}{\mathrm{APic}}

\EpigaVolumeYear{5}{2021} \EpigaArticleNr{8} \ReceivedOn{January 20,
2021}
\InFinalFormOn{February 8, 2020}
\AcceptedOn{February 16, 2021}

\title{The Cohen-Macaulay representation type of projective arithmetically Cohen-Macaulay varieties}
\titlemark{The CM-representation type of ACM varieties}

\author{Daniele Faenzi}
\email{daniele.faenzi@u-bourgogne.fr}
\address{Institut de Mathématiques de Bourgogne,
UMR CNRS 5584,
Université de Bourgogne et Franche-Comté,
9 Avenue Alain Savary,
BP 47870,
21078 Dijon Cedex,
France}

\author{Joan Pons-Llopis}
\email{juan.ponsllopis@polito.it}
\address{Dipartimento di Scienze Matematiche,
Politecnico di Torino,
Corso Duca degli Abruzzi 24,
10129 Torino, Italy}

\authormark{D.~Faenzi and J.~Pons-Llopis}

\AbstractInEnglish{We show that any reduced non-degenerate closed
subscheme $X\subset \PP^n$ of dimension $m \ge 
  1$ whose
graded coordinate ring is Cohen-Macaulay is of wild Cohen-Macaulay type, except for a few cases which we completely classify.}



\TitleInFrench{Le type de repr\'esentation de Cohen-Macaulay des variétés projectives arithm\'etiquement Cohen-Macaulay}

\AbstractInFrench{Nous montrons que tout sous-sch\'ema ferm\'e r\'eduit et non d\'eg\'en\'er\'e $X\subset \PP^n$ de dimension $m \geq 1$ dont l'anneau de coordonn\'ees homog\`enes est Cohen-Macaulay est de type de Cohen-Macaulay sauvage, except\'e dans quelques cas que nous classifions compl\`etement.}

\acknowledgement{Research supported by F\'ed\'eration de Recherche Bourgogne
  Franche-Comt\'e Math\'ematiques (FR CNRS 2011). D.~F. partially supported by ISITE-BFC project contract ANR-lS-IDEX-OOOB. J.~P.-L. partially supported by PRIN-2017 "Moduli Theory and Birational Classification".}

\begin{document}



\maketitle

\begin{prelims}

\DisplayAbstractInEnglish

\bigskip

\DisplayKeyWords

\medskip

\DisplayMSCclass

\bigskip

\languagesection{Fran\c{c}ais}

\bigskip

\DisplayTitleInFrench

\medskip

\DisplayAbstractInFrench

\end{prelims}


\newpage

\setcounter{tocdepth}{1}

\tableofcontents


\section{Introduction}

A classical result in representation theory of quivers is Gabriel's
theorem, stating that a finite connected quiver supports only finitely many irreducible
representations (that is, indecomposable modules over the
associated path algebra) if and only if it is of type $A$, $D$, $E$. The
classification of tame quivers as {\it Euclidean graphs}, or {\it
  extended Dynkin 
diagrams}, of type $\tilde A$,
$\tilde D$, $\tilde E$ came shortly afterwards. Remarkably, any other finite
connected quiver
supports arbitrarily large
families of indecomposable representations, which is to say, it is of {\it wild representation type}.

In algebraic geometry and commutative algebra, the relevant problem in
terms of representation theory of algebras concerns the complexity of the
category of maximal Cohen-Macaulay modules over the coordinate ring
$\bk[X]$ of a closed $m$-dimensional subvariety $X \subset \PP^n$ over a field
$\bk$. For $m>0$, assuming $\bk[X]$ to be Cohen-Macaulay (so $X$
is said to be arithmetically Cohen-Macaulay, briefly ACM), these
 correspond to ACM sheaves, namely coherent sheaves  $\cE$ on $X$ without intermediate
  cohomology, that is, satisfying $\HH^i(X,\cE(t))=0$ for all $t \in \ZZ$ and $0 < i < m$.
For hypersurfaces \cite{eisenbud:homological} these modules correspond to matrix
factorizations, which in turn are related to mirror symmetry,
see \cite{orlov:landau-ginzburg}.

In this sense, reduced projective ACM varieties of finite CM-type are
classified, see \cite{eisenbud-herzog:CM}, see also
\cite{buchweitz-greuel-schreyer,auslander-reiten:almost-split,knorrer:ACM,herzog:ringe-mit}.
Their list  (for positive dimension) consists of rational normal
curves, projective
spaces, smooth quadrics, the Veronese surface in $\PP^5$ and the
cubic scroll in $\PP^4$.
Of course this result is connected with Horrocks' and Grothendieck's classical splitting theorems
for vector bundles over $\PP^n$, which in turn relates to ideas
going back to Segre, \cite{segre:rigate-qualunque}. 

The next class consists of CM-tame varieties. These include
CM-countable varieties (for example quadrics of corank one) and varieties
where all indecomposable ACM sheaves are parametrized by a curve.
Besides
smooth elliptic curves (by seminal work of Atiyah,
\cite{atiyah:elliptic}, also related to classical work of Segre, \emph{cf.}
\cite{segre:rigate-ellittiche}), trees and cycles of rational curves
(see  \cite{drozd-greuel:CM-type,drozd-greuel}, see also
\cite{burban-drozd:coherent}), two sporadic examples were given in
\cite{faenzi-malaspina}, consisting of smooth  quartic surface scrolls
in $\PP^5$. 

The main goal of this paper is to prove that, besides theses cases, all
reduced closed ACM subschemes $X\subset \PP^n$ of positive dimension are
CM-wild. Without loss of generality, we may assume that $X$ is
non-degenerate, namely $X$ is
contained in no hyperplane.

\begin{thm*}
  Let $X\subset \PP^n$ be a reduced non-degenerate  closed ACM subscheme of dimension $m \ge
  1$.   Then $X$ is of wild CM-type unless $X$ is one of the following:
\begin{enumerate}[label=\rm{(\roman*)}]
  \item \label{linear} a linear space;
  \item \label{quadric} a quadric hypersurface of corank at most one;
  \item \label{rnc} a tree of rational
    curves;
  \item \label{elliptic} a smooth elliptic curve or a cycle of rational
    curves;
  \item \label{scroll} a smooth rational scroll of dimension $2$ and degree $d =
    3$ or $d=4$;
  \item \label{veronese} the Veronese surface in $\PP^5$.
  \end{enumerate}
\end{thm*}

As for the terminology used here, a \textit{rational scroll} is a variety
obtained as the image in 
$\PP^n$ of the projective bundle $Y=\PP(\oplus_{i=1}^m
\cO_{\PP^1}(a_i))$, for some integers $0\le a_1\le \cdots \le a_m \neq 0$, by the
relatively ample line bundle $\cO_Y(1)$. A rational scroll is smooth
if and only if $a_1>0$ or $a_{m-1}<a_m=1$ (in which case $Y=\PP^m$), otherwise it is a cone, see \S
\ref{extensions-cones}. As we will recall in a minute, rational
scrolls and quadrics form the class of varieties of \textit{minimal
  degree}, that play a rather special role in representation theory of algebras.

A \textit{tree of rational curves} is the union of distinct smooth rational curves
$X_1,\ldots,X_s$, namely each $X_i$ is isomorphic to $\PP^1$, such
that $X_i \cap X_j$
is a single point if $j \in \{i-1,i+1\}$ and empty otherwise.
A \textit{cycle of rational curves} can be either the same thing, but using cyclic
notation on the indices (so $X_1 \cap X_s \ne \emptyset$) or an
irreducible rational curve with a single ordinary double point.
This means
that the only singularities of the whole scheme $X$ are ordinary double points, so the
intersections points $X_{i-1} \cap X_i$ and $X_{i+1} \cap X_i$ are distinct.

\medskip

A word on the base field $\bk$ is in order. The result holds for an
algebraically closed field of arbitrary characteristic except
$2$. Actually all the results that we prove in this paper are valid
also in characteristic $2$. The only point where $\cha(\bk)\ne 2$ is
needed is when we recall the fact, due to Knörrer and Buchweitz-Greuel-Schreyer, that quadric cones of corank $1$ are
CM-countable. We refer to Remark \ref{char2} for a discussion of this
issue.

As a consequence of our main result, we get a strong version of the
finite-tame-wild trichotomy,
  namely that any reduced ACM closed subscheme $X \subset \PP^n$ of
  dimension $m>0$ falls in exactly one of the following
  classes:
  \begin{description}
  \item[\it Finite] there are only finitely many
    indecomposable ACM sheaves on $X$ up to isomorphism and degree shift. This
    happens in cases \ref{linear}, \ref{veronese},
    \ref{scroll} for $d=3$ and the smooth cases of \ref{quadric}, \ref{rnc}.
  \item[\it Tame] in turn also classically divided into \emph{tame discrete}: the parameter space of
    indecomposable non-isomorphic ACM sheaves is a countable set of points (in the singular cases of \ref{quadric}, \ref{rnc});
    or \emph{properly tame}: for any given rank $r$, the parameter space of
    indecomposable non-isomorphic ACM sheaves of rank $r$ is a finite union of curves (in cases
    \ref{scroll} for $d=4$ and \ref{elliptic}).
  \item[\it Wild] the category of modules of any finite-dimensional
    algebra admits a representation embedding into the category of
    MCM  $\bk[X]$-modules; in particular  $X$ supports families of
  arbitrarily large dimension of indecomposable non-isomorphic ACM sheaves.
\end{description}

The result was known for some specific cases,
such as smooth cubic surfaces (see
\cite{casanellas-hartshorne:ACM-cubic}), all linearly embedded Segre varieties besides
the CM finite ones (see \cite{costa-miro_roig-pons_llopis}),
smooth del Pezzo surfaces (see  \cite{miro-roig-pons-llopis:del-pezzo,
  coskun-kulkarni-mustopa:ulrich}), positive-dimensional hypersurfaces of degree at least
$4$ and some complete intersections (see \cite{drozd-tovpyha:JPAA}),
some Fano varieties (see \cite{miro-roig-pons-llopis:fano}), the triple
Veronese embedding of any variety (see \cite{miro-roig:representation-type:PAMS}).

\medskip

One should expect that non-projective varieties may behave
differently  (see
\cite{leuschke-wiegand:cohen-macaulay} for a detailed picture, see
also \cite{stone.non-gorensteint-countable}).
More varieties of tame type appear from germs of elliptic
singularities, see \cite{kahn:reflexive} or from non-isolated
affine surface singularities, see also \cite{burban-drozd:non-isolated}.

\medskip
Let us indicate the strategy
of our proof.
The first step is to
isolate some datum in order to build large families of
indecomposable non-isomorphic MCM modules on $\bk[X]$.
This will be accomplished by Theorem
\ref{strictly wild by extensions}, where we establish that this
datum should be a pair of ACM sheaves $\cA$ and $\cB$ on $X$, whose
only endomorphisms are homotheties (that is, $\cA$ and $\cB$ are
\textit{simple}), satisfying some semistability condition or a mutual
orthogonality condition $\Hom_X(\cB,\cA)=\Hom_X(\cA,\cB)=0$, and such that
$\Ext^1_X(\cB,\cA)$ is sufficiently large, namely of dimension $w \ge 3$.
From this datum we construct a representation embedding from
the category $\Rep_\Upsilon$
of finite-dimensional representations of the Kronecker quiver
$\Upsilon = \Upsilon_w$ with $w$ arrows to the
category $\MCM_{\bk[X]}$ of MCM modules over $\bk[X]$. This quiver should be seen
as parametrizing ACM sheaves appearing as extension of copies of $\cA$
and $\cB$.
In turn, by a standard argument, the existence of this embedding
suffices to prove CM-wildness
of $X$.
This procedure is not completely new, see \cite{drozd-greuel}, but we
endow it with a quite more general flavour.

The next step is to actually construct the sheaves $\cA$ and $\cB$.
This turns out to be quite complicated to achieve by working directly on $X$ in
general.
For this we need our next result,
Theorem \ref{stable FF}, which shows how to deduce
CM-wildness of $X$ from CM-wildness of $Y$ when $Y$ is a linear
section of $X$ of codimension $c>0$, except when $X$ has \textit{minimal degree}
that is $\deg(X)=n-m+1$, or equivalently when the sectional genus $p_X$ of $X$
is $0$.

In order to do this, we need to further assume that $\cA$ and
$\cB$ are Ulrich sheaves, namely their modules of global sections have
the maximal number of generators.
We see this as a further indication of the importance of these sheaves,
see \cite{eisenbud-schreyer-weyman,eisenbud-schreyer:betti-cohomology}.

The idea of Theorem \ref{stable FF} is that taking the $c^{\mathrm{th}}$ syzygy $\Om^c_{\bk[X]}$ of the
$\bk[X]$-module of an MCM module $L$ over $\bk[Y]$
one obtains an MCM module over $\bk[X]$ and that
this entails no essential loss of information if $L$ is Ulrich and $p_X>0$.
In fact,
for $\Om_{\bk[X]}^c$ to be a functor we need to pass to the stable category
$\SMCM_{\bk[X]}$ where we quotient out by morphisms factoring through
a free module. The point is that the stable syzygy
functor $\SOm^c$ is fully faithful on Ulrich modules when $p_X>0$.
The proof  uses cohomology
vanishing of Ulrich sheaves combined with duality.

The next result, Theorem \ref{embedding}, shows how to put these two
ingredients together. Indeed, by
resolving over $\bk[X]$ the module of global sections of the universal extension
of the sheaves $\cB$ by $\cA$ over $Y$ needed for Theorem \ref{strictly wild by extensions} and taking
its $c^\mathrm{th}$ syzygy, we get a functor $\Rep_\Upsilon \to \MCM_X$ whose
stabilization is fully faithful by Theorem \ref{stable FF}.
Then, although the functor itself is not quite fully faithful,
nevertheless it is a {\it representation embedding}, that is, it sends
non-isomorphic (respectively, irreducible) representations to non-isomorphic
(respectively, indecomposable) modules, and this suffices to show CM-wildness
of $X$.

In view of these results, in order to complete the proof it remains to treat directly the case $p_X=0$, and to construct the Ulrich sheaves $\cA$ and
$\cB$ as above over
a linear section $Y$ of $X$, which we take to be of dimension $1$
when $p_X\ge 2$, or of dimension
$2$ for $p_X = 1$.

The case $p_X \ge 2$ is rather easily seen to provide only CM-wild
varieties, as $\cA$ and $\cB$ can be taken to be sufficiently general
bundles of rank $2$ over $Y$ of slope $\deg(Y)+g-1$, where $g$ is the
geometric genus of $Y$. This is treated in \S
\ref{higher degree}.

For $p_X=1$ our proof of the existence of $\cA$ and $\cB$ is based on a  study of
locally free Ulrich sheaves of rank $2$ on surfaces of sectional genus
$1$, also called of \textit{almost minimal
degree}. Special care has to be taken to allow $Y$ to be singular and even
non-normal (yet neither reducible nor a cone, see the next paragraphs); nevertheless these varieties are completely classified and
sufficiently detailed information is available, in particular on their divisor class
group, to construct the required sheaves and to control their deformations. Theorem \ref{resumealmost} gathers the results of \S \ref{section almost minimal},
devoted to this case.

A different method is needed for reducible or non-reduced subschemes, since our
basic technique to construct the
sheaves $\cA$ and $\cB$ may fail for
various reasons. The two major ones are the following: first,
 the sheaves
$\cA$ and $\cB$ may
degenerate to non-simple ones when more components appear. Second, we
partially rely on the classification of varieties of
minimal and almost minimal degree, and for reducible subschemes this
is a far more complicated question than for irreducible ones.
We deal with this in Theorems \ref{two-pieces} and \ref{non-reduced}.

To summarize again, it  remains to work out the case of minimal
degree, which is equivalent to $p_X=0$.
This is easy if $X$ is smooth, but needs some care
if $X$ is singular, or equivalently if $X$ is a cone, which is to say,
the generators of the homogeneous ideal of $X$ do not 
involve a given set of variables (see \S \ref{extensions-cones}).
This case is not quite straightforward, mainly because
some smooth finite CM-type varieties
degenerate to singular ones that turn out to be CM-wild.
In \S \ref{minimal} we describe a
method to deal with cones and varieties of minimal
degree in a uniform manner.
This proves the stated CM-wildness of all cones except for a single
exceptional variety, namely
 the cone over a rational normal cubic. In turn, in Theorem
\ref{funnycase} we treat this
intriguing case with an
ad-hoc method based on representations of a certain quiver with three
vertices.

\medskip

If one carefully goes through the constructions carried out in this paper, it can be observed that many CM-wild varieties  actually support unbounded
families of Ulrich sheaves. A strong conjecture in this sense would be the following.

\begin{conj}
  Let $X \subset \PP^n$ be a closed non-degenerate integral subscheme of 
  dimension $m \geq 2$, not of minimal degree. Then $X$ is strictly Ulrich wild.
\end{conj}

Our main theorem offers an affirmative answer in case $X$ is ACM,
after replacing ``strictly Ulrich wild'' by ``CM-wild''.
The conjecture is known to hold for several classes of varieties,
most notably of surfaces, like del Pezzo surfaces or K3 surfaces
(see Theorem \ref{resumealmost} for surfaces of almost minimal degree,
which coincide with del Pezzo surfaces in the smooth case).
It is also true that curves of arithmetic genus greater or equal than two (see Section \ref{higher degree}) and smooth varieties of minimal degree of dimension $m
\ge 2$ are strictly
Ulrich wild except in cases \ref{linear}, \ref{quadric}, \ref{scroll}, \ref{veronese}
of our main theorem.
However this fails in general for singular varieties of minimal degree. For example, 
consider a quadric cone $X \subset \PP^n$ over a vertex $\Lambda$ of dimension
at least $1$. Then  $X$ is CM-wild.
On the other hand, an Ulrich sheaf on $X$ is the sheafification of
$E^0\otimes \bk[\Lambda]$, where $E^0$ is the module associated with a
direct sum of spinor bundles on a smooth quadric $X^0$,
the base of the cone. These sheaves are rigid, so $X$ is not
Ulrich-wild.

In a sense, the statement of the previous conjecture admits no converse, as it turns out that the Segre variety $\PP^1\times\PP^2\subset\PP^5$
is a CM-wild ACM variety whose only infinite family of
non-isomorphic indecomposable ACM sheaves (up a degree shift) consists of Ulrich bundles. We refer to \cite{faenzi-malaspina-sanna:non-ulrich} for this and related issues.

One may also observe that many singular CM-wild
varieties admit unbounded families of non-isomorphic ACM sheaves of
fixed rank, while this does seem not to happen for smooth varieties.
This motivates the following question.

\begin{prob}
 Let  $X \subset \PP^n$ be a smooth projective variety of positive
 dimension. Given $r \ge 1$, is the family of isomorphism classes of
 indecomposable ACM  initialized sheaves of rank $r$  parametrized by a finite union of
 irreducible quasi-projective schemes?
\end{prob}

The problem of classifying the representation type of integral subschemes
$X\subset \PP^n$ of dimension $m\ge 2$ which
are not ACM seems interesting. Some cases are known, such as abelian and
Enriques surfaces, (see \cite{casnati-ulrich} and \cite{beauv-abelian-ulrich}) but the general problem remains
wide open even for smooth surfaces. For reducible
subschemes, already the representation type of 2-regular subschemes
seems to be unknown in general.

\begin{akn}
  The second named author would like to thank the University of Bourgogne for the hospitality during the stay when part of this work was done.
\end{akn}

\section{CM-wild varieties}

Let $\bk$ be a field and set
$S=\bk[x_0,\ldots,x_n]$ for the symmetric graded $\bk$-algebra with $n+1$
indeterminates, seen also as the coordinate ring of the projective $n$-space
$\PP^n=\Proj(S)$ of $1$-dimensional linear quotients of the vector
space $V=\bk^{n+1}$.

\subsection{ACM varieties and modules}

We first recall some basic terminology for various Cohen-Macaulay
properties of varieties, sheaves and modules.

\subsubsection{ACM subschemes}

Let $X \subset \PP^n$ be a closed subscheme of dimension $m>0$.
Write $I_X$ for the saturated homogeneous ideal of $X$ and $R=\bk[X]=S/I_X$ for its coordinate ring.

\begin{dfn}
 The subscheme $X \subset \PP^n$ is
  arithmetically Cohen-Macaulay (ACM) if $R=\bk[X]$ is a graded
  Cohen-Macaulay ring,
  namely if $R$ has a graded $S$-free
  resolution of length $n-m$.
\end{dfn}

\subsubsection{Terminology on coherent sheaves}

Let $X \subset \PP^n$ be a closed subscheme, $m=\dim(X)$.
We write $\Coh_X$ for the category of coherent sheaves on $X$.
We denote by $\cO_X(1)$ the restriction to $X$ of
$\cO_{\PP^n}(1)$ and we employ the usual notation $\cE(t)=\cE \otimes
\cO_X(1)^{\otimes t}$.
The ideal sheaf of a subscheme $Z \subset X$ will be denoted by $\cI_{Z|X}$.
Given  $\cE,\cF$ in $\Coh_X$ and $i \in \ZZ$, we consider the Ext modules:
\[
\Ext^i_{X}(\cE,\cF)_*= \bigoplus_{t \in \ZZ}
\Ext^i_{X}(\cE,\cF(t))
\]
as
$R$-modules. For $i\in \NN$, we write
$\HH^i_*(\cE)=\Ext^i_X(\cO_X,\cE)_*$ for the $i^\mathrm{th}$ cohomology module of
$\cE$. One may also replace $t\in \ZZ$ by any truncation $t \ge t_0$.
The module $\HH^0_*(\cE)$ is also denoted by $\Gamma_*(\cE)$.
It is finitely generated if $\cE$ has no zero-dimensional
subsheaf.

\medskip

We say that a coherent sheaf $\cE$ on $X$ is {\it simple} if its
only endomorphisms are homotheties, that is, if $\Hom_X(\cE,\cE) =
\bk \id_\cE$. We write $\chi(\cE,\cF)$ for the Euler characteristic of two coherent sheaves $\cE$
and $\cF$ over $X$, namely
$$
\chi(\cE,\cF)=\sum_{i \in \ZZ}(-1)^i\dim_\bk\Ext^i_X(\cE,\cF),
$$
provided this is a finite sum. This is the case for instance when $X$ is smooth or when $\cE$ or $\cF$ are locally
free. We abbreviate $\chi(\cF)=\chi(\cO_X,\cF)$.

\medskip

We write $H_X$ for the very ample divisor class on $X$ associated with $\cO_X(1)$.
The {\it Hilbert polynomial} of a coherent sheaf $\cE$ is defined as
$P(\cE,t):=\chi(\cE(t))$. The degree $d=\deg(X)$ is defined in terms
of the polynomial $P(\cO_X,t)$, namely by the condition that the
leading term of $P(\cO_X,t)$ be $d/m!$. Similarly, for $\cE$ in
$\Coh_X$, the rank $r=\rk(\cE) \in \QQ$ is defined by the condition
that the leading term of $P(\cE,t)$ be $r d/m!$.
We write $\pp(\cE,t):=P(\cE,t)/r$ for the {\it reduced
  Hilbert polynomial}.

We write $p \succeq q$ (resp. $p \succ q$) for polynomials $p,q \in \QQ[t]$ if $p(t) \ge
q(t)$ (resp. $p(t) > q(t)$) for $t \gg 0$. A coherent sheaf $\cE$ is
called {\it pure} if all of its subsheaves are supported in
 dimension $m$. A pure sheaf is $H_X$-{\it
  semistable} in the sense of Gieseker-Maruyama if, for any coherent
subsheaf $\cF \subsetneq \cE$, one has $\pp(\cE,t)\succeq \pp(\cF,t)$. The sheaf is called
$H_X$-stable if for all $\cF$ as
above $\pp(\cE,t) \succ \pp(\cF,t)$. We will often suppress $H_X$ from the notation.

\subsubsection{Cohen-Macaulay and Ulrich conditions}
Again $X \subset \PP^n$ is a closed subscheme of dimension $m \ge 1$ with
coordinate ring $R=\bk[X]$. Given a graded $R$-module $M$ and
$t\in\ZZ$, we denote by $M_t$ its degree-$t$ piece and $M_{\ge
  t}=\oplus_{i\ge t}M_i$. Analogously $M_{< t}=\oplus_{i < t}M_i$.

\begin{dfn}
  A coherent sheaf $\cE$ on an $m$-dimensional closed subscheme
  $X\subset \PP^n$ is
  called \emph{arithmetically Cohen-Macaulay (ACM)} if $\cE$ is
  locally Cohen-Macaulay (that is, $\cE_x$ is a Cohen-Macaulay
  $\cO_{X,x}$-module for all $x\in X$)  and $\HH^i_*(\cE)=0$ for
  $i=1,\dots,m-1$.
\end{dfn}

This is equivalent to asking that $E=\Gamma_*(\cE)$ is
a maximal graded Cohen-Macaulay module over $R$. In turn, this amounts
to requiring that $E$ has a graded free $S$-resolution of length $n-m$.

Let $d=\deg(X)$. Given an MCM module $E$ of rank $r$ over $R=\bk[X]$, the number of independent minimal generators of
$E$ is at most $d r$. Analogously, for an ACM sheaf $\cE$, assuming that $\cE$ is {\it initialized}  (\textit{i.e.},  $\HH^0(X,\cE)>\HH^0(X,\cE(-1))=0$), we have:
\begin{equation}
  \label{atmost}
\dim_{\bk}\HH^0(X,\cE) \le d r.
\end{equation}

An ACM coherent sheaf $\cE$ on $X$ is called an {\it Ulrich sheaf} on $X$ (and
$E=\Gamma_*(\cE)$ is called an Ulrich module over $R$) if $\HH^0(X,\cE(-1))=0$ and equality is attained in \eqref{atmost}.
The reader can consult \cite{eisenbud-schreyer-weyman} for an account
on Ulrich sheaves. Let us just gather here the main properties that will be used throughout this paper:
\begin{enumerate}[label=\alph*)]
  \item Any Ulrich sheaf $\cE$ of rank $r$ on an $m$-dimensional closed subscheme
  $X\subset \PP^n$ of degree $d$ has a linear $\cO_{\PP^n}$-resolution of the form
  $$
  0 \la \cE \la \cO_{\PP^n}^{d r} \lx{d_1} \cO_{\PP^n}(-1)^{a_1} \la
  \cdots\lx{d_{n-m}} \cO_{\PP^n}(m-n)^{a_{n-m}}\la 0.
  $$
  The length of the resolution is $n-m$ and the maps $(d_i\mid i\in
  \{1,\ldots,n-m\})$  are given
  by matrices whose entries are linear forms of
  $S$. Also one has $a_i=\binom{n-m}{i}d r$ for all  $i \in \{1,\ldots,n-m\}$
  This follows from \cite[Proposition 2.1]{eisenbud-schreyer-weyman},
  see also the comments after this proposition.

  \item Any Ulrich sheaf $\cE$ is globally generated. Its Hilbert
    polynomial is $P(\cE,t)=d r\binom{t+m}{m}$. This is a consequence
    of the previous point.
  \item For any linear projection $\pi:X\rightarrow \PP^m$, the direct
    image $\pi_*\cE$ is isomorphic to $\cO_{\PP^m}^{d r}$. Again \cite[Proposition 2.1]{eisenbud-schreyer-weyman}.
  \item Any $\cE$ Ulrich sheaf on a subscheme  $X \subset \PP^n$ is
    semistable and any destabilizing  subsheaf of $\cE$ is also
    Ulrich.
    This has been proved in a number of papers, with variable hypothesis; see for instance
    \cite[Theorem 2.9]{casanellas-hartshorne-geiss-schreyer} for smooth varieties. We refer to Lemma \ref{ulrich
      semistable} for a statement on an arbitrary closed subscheme.
\end{enumerate}

Let us denote by $\ACM_X$ (resp. $\Ulr_X$) the full subcategory of $\Coh_X$ consisting
of ACM sheaves (resp. of Ulrich sheaves). We denote by $\MCM_{R,0}$
(resp. $\Ulr_{R,0}$) the subcategory of the category $\Mod_R$ of finitely
generated $R$-modules whose objects are MCM
modules (resp. Ulrich
modules) and whose morphisms are degree-$0$ morphisms of
$R$-modules. There is a basic equivalence between these notions as in the next
lemma (see \cite[Proposition 2.2.4]{kleiman-landolfi}).

\begin{lem}\label{equiv-mcm-acm}
  The functor $\Gamma_* : \ACM_X \to \MCM_{R,0}$ is an equivalence, whose
  inverse is the sheafification functor $M \mapsto \tilde M$. The
  equivalence carries $\Ulr_X$ to $\Ulr_{R,0}$.
\end{lem}

Let $\cE$ and $\cF$ be coherent sheaves on $X$ whose
associated modules $E=\Gamma_*(\cE)$ and $F=\Gamma_*(\cF)$ are finitely generated as $R$-modules.
In spite of the previous lemma, $\Ext^i_X(\cE,\cF(t))$ and
$\Ext^i_R(E,F)_t$ may differ. The following lemma will be useful to compare them.

\begin{lem} \label{Ext-Ext}
Assume $\cF$ is ACM. Then, there is $t \in \ZZ$
depending on the minimal graded free resolution of $F$ as $S$-module, such that
there is an isomorphism:
\[
\Ext^i_R(E_{\ge t},F)_{\ge 0} \simeq \bigoplus_{q \ge 0}
\Ext^i_{X}(\cE,\cF(q)), \qquad \mbox{as graded $R$-modules}.
\]
\begin{enumerate}[label=\roman*)]
\item \label{i le m-1} If $i \le m-1$, in the above isomorphism we may replace $E_{\ge t}$ by $E$.
\item \label{if F Ulrich} If $F$ is Ulrich, we may take $t=1-i$.
\item \label{if both Ulrich} If $F$ is linearly presented and $E$
  is generated in degree $0$ over $S$, then $\Ext^1_R(E,F)_{< -1}=0$.
\end{enumerate}
\end{lem}

\begin{proof}
  Since the sheaf $\cF$ is ACM of positive
   dimension, the associated $R$-module $F$ is finitely generated.
   The first statement follows from \cite[Theorem 1]{smith-extension-modules}. If $F$ (or $\cF$) is Ulrich,  the minimal
   graded free resolution of $F$ as $S$-module is linear, so the same result
   allows $t=1-i$, so we get \ref{if F Ulrich}.

   Next, $F$ is a graded Cohen-Macaulay $R$-module, so
   $\Ext^i_R(\bk,F)=0$ for $i\le m$, where $\bk$ is the residue
   field seen as a $R$-module.
  Then, for any $j \in \ZZ$, since the module $E_{<j} := E/E_{\ge j}$ is Artinian, by
  induction on the length of the composition series of $E_{<j}$ we get
  $\Ext^i_R(E_{<j},F)=0$ for $i\le m$.
  The isomorphism needed for \ref{i le m-1} follows from the exact sequence:
  \[
  \Ext^i_R(E_{<j},F) \to \Ext^i_R(E,F) \to
  \Ext^i_R(E_{\ge j},F) \to \Ext^{i+1}_R(E_{<j},F).
  \]

  It remains to prove \ref{if both Ulrich}. A minimal graded presentation
   of $F$ as $S$-module is an exact sequence of the form
   $S^{\beta_1}(-1) \to S^{\beta_0} \epi F$, while $E$ admits a
   surjection $S^{\alpha_0} \epi E$,  for some positive integers
   $\alpha_0,\beta_0$ and $\beta_1$.
   The surjection $S^{\beta_0} \to F$, whose kernel is an $S$-module
   that we denote by $N$, descends to a surjection of
   $R$-modules $R^{\beta_0} \to F$, whose kernel we call $M$. Using $\Hom_R(-,E)$ we get an exact
   sequence of $R$-modules:
   \[
   \Hom_R(R^{\beta_0},E) \to \Hom_R(M,E) \to \Ext^1_R(F,E) \to 0.
   \]
   Now, $M$ is the quotient of $N\otimes _S R$ by $\Tor_1^S(R,F)$, so
   we have surjections
   $R(-1)^{\beta_1} \epi N \otimes_S R \epi M$.
   We deduce that $\Hom_R(M,E)$ is a submodule of
   $\Hom_R(R(-1)^{\beta_1},E)$.
   From the surjection $S^{\alpha_0} \epi E$ we get $R^{\alpha_0} \epi
   E$ and thus a surjection:
   \[
   \Hom_R(R(-1)^{\beta_1},R^{\alpha_0}) \epi \Hom_R(R(-1)^{\beta_1},E).
   \]
   This proves that $\Hom_R(R(-1)^{\beta_1},E)$ vanishes in degree
   strictly below $-1$ so the same
   happens to $\Hom_R(M,E)$ and thus to $\Ext^1_R(F,E)$.
\end{proof}

\subsection{CM-wildness}

We will consider a couple of related notions of CM-wildness for a
closed scheme $X\subset\PP^n$. 
 Algebraically this means that, for any
finitely generated associative $\bk$-algebra
$\Sigma$, the category of MCM modules over $R=\bk[X]$ contains, in some
sense, the category $\Mod_\Sigma$ of finitely generated left
$\Sigma$-modules. We spell this out in more detail in the
next paragraph. We adopt
 \cite[Chapter XIX]{simson-skowronski} as general reference for this part.

\begin{dfn}
  Let $X \subset \PP^n$ be a closed subscheme and set $R=\bk[X]$.
  For any finitely generated associative
  $\bk$-algebra $\Sigma$, and any finitely generated $R$-graded
  $(R,\Sigma)$-bimodule $\bM$, flat 
  over $\Sigma$, define the functor:
\begin{align*}
    \Phi_{\bM} : \Mod_\Sigma &\to \Mod_R, \\
    N &\mapsto \bM \ts _\Sigma N.
  \end{align*}
  The variety $X$ is
  of {\it wild CM-type} if, for any $\Sigma$ as above, there is $\bM$
  such that $\Phi_{\bM}$ takes values in $\MCM_R$ and is a {\it
    representation embedding} in $\MCM_{R}$, which is to say:
  \begin{enumerate}[label=\alph*)]
  \item  the module $N$ is decomposable whenever $\Phi_{\bM}(N)$ is;
  \item for any pair $(N,N')$ of modules in $\Mod_\Sigma$, we have:
    \[
N \simeq N'  \Leftrightarrow   \Phi_{\bM}(N)  \simeq \Phi_{\bM}(N').
    \]
  \end{enumerate}
  The variety $X$ is of {\it wild Ulrich type} if moreover:
\begin{enumerate}[resume, label=\alph*)]
  \item for any $N$ in $\Mod_\Sigma$, $\Phi_{\bM}(N)$ is Ulrich.
\end{enumerate}

The variety $X$ is said to be {\it strictly CM-wild} if for any $\Sigma$
as above there is an $\bM$ such that $\Phi_{\bM}$ is fully faithful into
$\MCM_{R,0}$, that is
  \[
  \Hom_\Sigma(N,N') \simeq \Hom_R(\Phi(N),\Phi(N'))_0,
  \]
If moreover $\Phi_{\bM}(N)$ is
Ulrich for all $N$, then $X$ is  {\it strictly Ulrich wild}.
\end{dfn}

\begin{rmk} \label{well known}
The following facts are well-known, or quickly proved in the next lines.
\begin{enumerate}[label=\roman*)]
\item To check that $X$ is CM-wild, it is enough to construct a
  representation embedding from $\FMod_\Sigma$ into $\MCM_R$, where
  $\Sigma$ is a wild algebra of finite dimension over $\bk$ and $\FMod_\Sigma$ is
  the category of $\Sigma$-modules of finite dimension over $\bk$.
  Thanks to Lemma \ref{equiv-mcm-acm}, we may work
  interchangeably with
  $\ACM_X$ or $\MCM_{R,0}$.
\item If all non-zero graded $R$-modules $\Phi(M)$ are generated in the same
  degree and $X$ is strictly CM-wild, then it is CM-wild, and if $X$
  is strictly Ulrich wild then it is of wild Ulrich type.
  Indeed, given $M,N \in \Mod_\Sigma$, since $\Phi(M)$ and $\Phi(N)$ are generated in the same
  degree, an isomorphism $\Phi(M) \to \Phi(N)$ must be of degree $0$
  and therefore must come from an isomorphism $M \to N$ by full
  faithfulness. Also any idempotent of $\Phi(M)$ must have degree
  $0$ and is therefore induced by an idempotent of $M$.
\item If $X$ is of wild CM-type, then for any $r \in \NN$ there are
  families of dimension at least $r$ consisting of indecomposable ACM
  sheaves on $X$, all non-isomorphic to
  one another. In other words, $X$ is of wild CM-type in the geometric sense. If
  $X$ is of wild Ulrich type, these families can be taken to consist of Ulrich sheaves.
\item Any  exact functor $\Phi:\Mod_\Sigma
  \to \Mod_R$ is of the form $\Phi_{\bM}$ for some finitely generated
  $\Sigma$-flat $(R,\Sigma)$-bimodule $\bM$.
\end{enumerate}
Let $w \ge 1$ be an integer and consider the Kronecker quiver
$\Upsilon=\Upsilon_w$ with two vertices and $w$ arrows from the first vertex to
the second. Write $\Rep_\Upsilon$ for the abelian category of
finite-dimensional $\bk$-representations of $\Upsilon$.
\begin{enumerate}[resume, label=\roman*)]
\item To check that $X$ is strictly CM-wild (resp., of wild CM-type), it suffices to construct
  a fully faithful exact functor (resp., a representation embedding):
  \[
  \Phi : \Rep_\Upsilon \to \ACM_X
  \]
  where $\Upsilon=\Upsilon_w$ is the Kronecker quiver
  with $w \ge 3$. If moreover $\Phi(\cR)$ is Ulrich for
  any $\cR$ in $\Rep_\Upsilon$, then $X$ is strictly
  Ulrich wild (resp., of wild Ulrich type).
  The same argument works if we replace $\Rep_\Upsilon$ with
  $\FMod_\Sigma$ where $\Sigma=\bk[x_1,x_2]$.
\end{enumerate}
\end{rmk}

\section{CM-wildness from extensions}

Let $X\subset \PP^n$ be a closed $\bk$-subscheme, and let $\cA$ and $\cB$ be
coherent sheaves on $X$ such that:
\[
\Ext^1_X(\cB,\cA) \ne 0.
\]
We describe how extensions of $\cB$ by $\cA$ are parametrized by
representations of the \emph{Kronecker quiver} $\Upsilon_w$ having two
vertices and as many arrows as $w=\dim_{\bk} \Ext^1_X(\cB,\cA)$, pointing in the
same direction.

\subsection{The functor from the Kronecker quiver to $\Coh_X$} \label{kronecker-quiver}\label{extensions}

Set $W=\Ext^1_X(\cB,\cA)$ and consider the projective space $\PP(W^*)$ of
lines through the origin in $W$.
Then, over $X \times \PP (W^*)$, there is
a universal extension:
\[
0 \to \cA \boxtimes \cO_{\PP(W^*)} \to \cU \to \cB \boxtimes \cO_{\PP(W^*)}(-1) \to 0,
\]
where we write  $p$ and $q$ for the projections from $X \times \PP (W^*)$
to $X$ and $\PP (W^*)$, and for $\cE \in \Coh_X$ and $\cF \in
\Coh_{\PP (W^*)}$, we set $\cE \boxtimes \cF = p^*(\cE) \ts q^*(\cF)$.
Then we consider:
\[
\Phi_\cU=\bR p_* (q^*(-) \ts \cU) : \bD(\Coh_{\PP (W^*)}) \to \bD(\Coh_X).
\]
It is clear that:
\[
\Phi_\cU(\cO_{\PP (W^*)}) \simeq \cA, \qquad \Phi_\cU(\Omega_{\PP (W^*)}(1)) \simeq \cB[-1].
\]

Set $w= \dim_{\bk} \Ext^1_X(\cB,\cA)$ and consider the Kronecker quiver $\Upsilon =
  \Upsilon_w$.
Then, the natural isomorphism $W \simeq \Hom_{\PP (W^*)}(\Omega_{\PP
  (W^*)}(1),\cO_{\PP (W^*)})$ provides an equivalence:
\begin{equation}
\label{Xi}
\Xi : \bD(\Rep_\Upsilon) \simeq \langle \Omega_{\PP (W^*)}(1),\cO_{\PP (W^*)} \rangle.
\end{equation}
We compose this equivalence with the inclusion of $\langle \Omega_{\PP
  (W^*)}(1),\cO_{\PP (W^*)} \rangle$ into $\bD(\Coh_{\PP (W^*)})$.
Explicitly, this is described as follows.
Choose a basis $(e_1,\ldots,e_w)$ of $W=\Ext^1_X(\cB,\cA)$.
Let $\cR$ be a representation of $\Upsilon$ having dimension vector
$(a,b)$.
Then $\cR$ corresponds to the choice of $w$ linear maps $m_1,\ldots,m_w:\bk^a \to \bk^b$.
Take the element:
\begin{equation}
  \label{xi}
  \xi=\sum_{i=1}^w m_i\otimes e_i\in  \Hom_{\bk}(\bk^b,\bk^a) \ts  W.
\end{equation}
Then, under the identification $W \cong \Hom_{\PP (W^*)}(\Omega_{\PP
  (W^*)}(1),\cO_{\PP (W^*)})$, we obtain from $\xi$ a morphism:
\begin{equation}
  \label{M}
  M :\Omega_{\PP  (W^*)}(1)^b \to \cO_{\PP (W^*)}^a.
\end{equation}
The cone of $M$ is the element of $\bD(\Coh_{\PP (W^*)})$ associated with $\cR$
via $\Xi$. This is directly extended to morphisms.

We consider a functor $\Phi$ which can be thought of as the restriction of $\Phi_\cU \circ \Xi$ to $\Rep_\Upsilon$:
\begin{equation}\label{phi}
  \Phi : \Rep_\Upsilon \to \Coh_X.
\end{equation}

 Let us first give an explicit description of $\Phi$.  At the level of objects, given a representation $\cR$ of the quiver $\Rep_\Upsilon$ with dimension vector $(a,b)$
  let $(m_1,\ldots,m_w)$ be the $w$ linear maps associated with $\cR$ and let $\xi$ be as in \eqref{xi}.  Then $\Phi(\cR)$ fits as
  middle term of a
  representative of the extension class
  corresponding to $\xi$:
  \[
  0\to\cA^a\to\Phi(\cR)\to\cB^b\to 0.
  \]

  Let us check now that this is well-defined on morphisms. Let $\cS$ be another representation
  of $\Upsilon$, of dimension vector $(c,d)$, corresponding to the
  linear maps $(n_1,\ldots,n_w)$.
  A morphism $\lambda:\cR\to\cS$ of representations is given by linear
  maps $\alpha:\bk^a\to\bk^c$ and $\beta:\bk^b\to\bk^d$ such that:
  \begin{equation}
    \label{commutes}
    n_i\alpha=\beta m_i, \qquad \mbox{for all $i=1,\ldots,w$.}
  \end{equation}

  Consider the map of coherent sheaves $\beta_\cA=\beta \ts
  \id_\cA:\cA^a\to\cA^c$.
  Then, $\beta_\cA$ defines a morphism of extensions:
  \begin{equation}
    \label{diag0}
    \xymatrix@-2ex{
      0 \ar[r] &  \cA^a \ar[d]^-{\beta_\cA} \ar[r] &  \Phi(\cR)  \ar[d]^{\phi} \ar[r] &  \cB^b \ar@{=}[d] \ar[r] & 0\\
      0 \ar[r] &  \cA^c  \ar^-i[r] & \cD \ar^p[r] & \cB^b  \ar[r] & 0
    }
  \end{equation}
  for a certain sheaf $\cD$ representing the element:
  \[
  \sum_{i=1}^w \beta m_i\otimes e_i\in \Ext^1_X(\cB^b,\cA^c).
  \] Analogously  $\alpha_\cB = \alpha \ts \id_\cB:\cB^b\to\cB^d$ defines:
  \begin{equation}
    \label{diag2}
    \xymatrix@-2ex{
      0 \ar[r] &  \cA^c \ar@{=}[d] \ar^{i'}[r] &  \cD'  \ar^-{\phi'}[d] \ar^{p'}[r] &  \cB^b \ar[d]^{\alpha_\cB} \ar[r] & 0\\
      0 \ar[r] &  \cA^c  \ar[r] & \Phi(\cS) \ar[r] & \cB^d  \ar[r] & 0
    }
  \end{equation}
  with the upper row representing an extension class in:
  \[\sum_{i=1}^w n_i\alpha\otimes e_i\in \Ext^1_X(\cB^b,\cA^c).\]

  Because of \eqref{commutes}, the lower extension of \eqref{diag0}
  is the same as the upper one in \eqref{diag2}. Then the morphisms $\phi'$ and $\phi$
  of extensions compose to give a map from $\Phi(\cR)$ to $\Phi(\cS)$.

  This construction agrees with the principle of considering
  $\Phi$ as restriction of $\Phi_\cU \circ \Xi$ to $\Rep_\Upsilon$.
  Indeed, the representation $\cR$ is mapped by
  $\Xi$ to the cone of the matrix $M$ of \eqref{M}, which is
  sent by $\Phi_\cU$ to the cone of:
  \[
  \Phi_\cU(M):\cB^b[-1] \to \cA^a.
  \]
  By construction this cone is  represented by the extension class $\Phi(\cR)$.

\subsection{Representation embeddings of the Kronecker quiver}

Now we state a basic result on representation embeddings and fully
faithful embeddings of the Kronecker quiver via extensions.
Some forms of this result have been used already by several authors,
however the following rather general statement seems to be new.

\begin{THM} \label{itswild} Let $\cA$ and $\cB$ be simple coherent sheaves on a
  closed subscheme $X \subset \PP^N$. \label{strictly wild by extensions}
  \begin{enumerate}[label=\roman*)]
  \item \label{nonstrict} Let $\cA$ and $\cB$ be semistable with $\pp(\cB) \preceq
    \pp(\cA)$ and suppose that any non-zero morphism $\cA \to \cB$ is
    an isomorphism. Then the functor $\Phi$  from (\ref{phi}) is a representation embedding.
  \item \label{strict} Assume that $\Hom_X(\cA,\cB)=\Hom_X(\cB,\cA)=0$.
  Then the functor $\Phi$ is fully faithful.
  \end{enumerate}
\end{THM}

\begin{proof}
  To check \ref{nonstrict} we first prove that, given an irreducible representation $\cR$ of $\Upsilon$, the
  associated sheaf $\cF=\Phi(\cR)$ is indecomposable. Let $\cR$ have
  dimension vector $(b,a)$ so we have:
\begin{equation}
  \label{AB}
0 \to \cA^a \rx{i} \cF \rx{p} \cB^b \to 0.
\end{equation}

  Assume first $\pp(\cB) \prec \pp(\cA)$. Then, \eqref{AB} is the Harder-Narasimhan
  filtration of $\cF$, so the graded object $\gr(\cF)$ associated with $\cF$ is
  just $\cA^a \oplus \cB^b$. Assume $\cF \simeq \cF' \oplus \cF''$,
  with $\cF' \ne 0 \ne \cF''$. By the uniqueness of $\gr(\cF)$, we
  have $\gr(\cF') \simeq \cA^{a'} \oplus \cB^{b'}$ and
  $\gr(\cF'') \simeq \cA^{a''} \oplus \cB^{b''}$
  for some $(b',a')$
  and $(b'',a'')$ with $a'+a''=a$ and $b'+b''=b$.
  It follows that $\cA^{a'}$ is the maximal destabilizing subsheaf of
  $\cF'$ with quotient $\cB^{b'}$, that is, $\cF'$ is an extension of the
  form:
  \[
  0 \to \cA^{a'} \to \cF' \to \cB^{b'} \to 0,
  \]
  associated with some $\xi' \in W \ts \bk^{a'} \ts
  \bk^{b'}$. Similarly, there is $\xi'' \in W \ts \bk^{a''} \ts
  \bk^{b''}$ corresponding to $\cF''$.
  Moreover, composing the embedding $\cA^{a'} \mono \cF'$  with $\cF'
  \mono \cF$ we
  get a map $\cA^{a'} \mono \cF$, that composes to zero with $p$, for
  $\cA$ and $\cB$ are semistable with $\pp(\cB) \prec \pp(\cA)$.

  We obtain thus a map $\cA^{a'} \mono \cA^a$ which must be of the
  form $\alpha \ts \id_{\cA}$ for some monomorphism $\alpha : \bk^{a'}
  \to \bk^a$, because $\cA$ is simple.
  This induces a map $\cB^{b'} \to \cB^b$ which is likewise of the
  form $\beta \ts \id_\cB$ for some $\beta : \bk^{b'}
  \to \bk^b$.
  This defines a representation $\cR'$ of dimension vector $(b',a')$
  corresponding to $\xi'$ which is
  a subrepresentation of $\cR$, the embedding being given by
  $(\beta,\alpha)$. The quotient $\cR''=\cR/\cR'$ corresponds 
  then to $\xi''$. The embedding $\cF'' \mono \cF$ provides a
  splitting $\cR'' \to \cR$, so $\cR \simeq \cR' \oplus \cR''$, with
  $\cR' \ne 0 \ne \cR''$ which
  is what we wanted.

  Now assume $\pp(\cB)=\pp(\cA)$ and take $\cR$ indecomposable such that
  $\cF = \cF'
  \oplus \cF''$ with $\cF' \ne 0 \ne \cF''$.
  Then \eqref{AB} is a Jordan-Hölder
  filtration of $\cF$, so the graded object $\gr(\cF)$ associated with $\cF$ is
  again $\cA^a \oplus \cB^b$, hence $\gr(\cF')$ and
  $\gr(\cF'')$ take the same forms as above, in particular $\cF'$ and
  $\cF''$ are semistable with $\pp(\cF')=\pp(\cA)=\pp(\cF'')$.

  Next, we compose $i$ with the projection $\cF \to \cF''$ to get a map
  $q:\cA^a \to \cF''$. The sheaf $\im(q)$ is semistable with
  $\pp(\im(q))=\pp(\cA)$. Composing with the projection to
  $\cB^{b''}$, since any non-zero map $\cA \to \cB$ is an isomorphism,
  we get as image a direct sum of copies of $\cA$. So $\im(q)$
  projects onto copies of $\cA$ and thus, since $\cA$ is simple, we
  actually have $\im(q)
  \simeq \cA^{a''}$ for some integer $a''$, which also gives $\ker(q)
  \simeq \cA^{a'}$ with $a'=a-a''$. By the same argument, composing
  the injection $\cF' \mono \cF$ with $p$ gives a map $j$ whose image
  is $\cB^{b'}$, for some integer $b'$, and whose cokernel is then
  $\cB^{b''}$ with $b''=b-b'$.
  Using that $\cA$ and $\cB$ are simple, we finally get a commutative exact diagram:
  \[
    \xymatrix@-2ex{
      0 \ar[r] &  \cA^{a'} \ar_-{\alpha'\ts \id_\cA}[d] \ar[r] & \cF' \ar[d] \ar^-p[r] &  \cB^{b'} \ar^-{\beta'\ts \id_\cB}[d] \ar[r] & 0\\
      0 \ar[r] &  \cA^a  \ar_-{\alpha''\ts \id_\cA}[d] \ar[r] & \cF \ar[r] \ar[d] & \cB^b \ar^-{\beta''\ts \id_\cB}[d]  \ar[r] & 0 \\
      0 \ar[r] &  \cA^{a''}  \ar[r] & \cF'' \ar^-p[r] &  \cB^{b''} \ar[r] & 0
    }
    \]
    for some maps $\alpha' : \bk^{\alpha'} \to \bk^\alpha$, $\alpha''
    : \bk^{\alpha} \to \bk^{\alpha''}$, and similarly for $\beta$.
  This says that there are representations $\cR' \ne 0 \ne \cR''$ of
  $\Upsilon$ with dimension vectors $(b',a')$ and $(b'',a'')$ such
  that $\cF' \simeq \Phi(\cR')$ and $\cF''\simeq \Phi(\cR'')$. Also,
  the maps $\alpha'$, $\alpha''$, $\beta'$, $\beta''$ provide a exact
  sequence:
  \[
  0 \to \cR' \rx{(\beta',\alpha')} \cR \rx{(\beta'',\alpha'')} \cR'' \to 0.
  \]
  Using the splitting map $\cF \to \cF'$ we see that $\cR = \cR'
  \oplus \cR''$, which is what we needed.

  Finally, we would like to show that two representations $\cR$ and
  $\cS$ of
  $\Upsilon$ are isomorphic if and only if their images via $\Phi$ are.
  Let $\cR$ and $\cS$ have dimension
  vectors $(b,a)$ and $(d,c)$. We may suppose that $\cR$ and $\cS$ are
  irreducible, so that $\cF=\Phi(\cR)$ and $\cG=\Phi(\cS)$ are indecomposable, by the
  first part of the proof.
  Take an isomorphism $\phi : \cF \to \cG$.
  Composing $\phi$ on the left  with the injection
  $\cA^a \to \cF$ and on the right with the projection $\cG \to
  \cB^b$, we get a map $\phi_0$.
  We distinguish two cases according to whether $\phi_0$ is zero or not.

  In the latter case, there is a summand
  $\cA$ of $\cA^a$ that maps non-trivially, hence isomorphically, to a summand $\cB$ of
  $\cB^d$. We deduce that $\cB$ is a direct summand of $\cG$. By the
  assumption on irreducibility of $\cS$, $\cG \simeq \cB$, which
 gives the conclusion.

  In the former case, we get a map $\alpha_\cA : \cA^a \to \cA^c$ inducing an exact
  commutative diagram:
  \[    \xymatrix@-2ex{
      0 \ar[r] &  \cA^a \ar^-{\alpha_\cA}[d] \ar^-i[r] &
      \cF  \ar[d]^-{\phi} \ar^-p[r] &  \cB^b \ar^-{\beta_\cB}[d] \ar[r] & 0\\
      0 \ar[r] &  \cA^c  \ar[r] & \cG \ar[r] & \cB^d  \ar[r] & 0
    }
  \]
  Hence $\cok (\alpha_\cA) \simeq \cA^{c-a} \simeq \cB^{b-d} \simeq
  \ker(\beta_\cB)$. Then $a \le c$ and $d \le b$. But using
  $\phi^{-1}$ we get the opposite inequalities, which implies that $\alpha_\cA$
  and $\beta_\cB$ are isomorphisms.
  Again $\alpha_\cA = \alpha \ts \id_\cA$ and $\beta_\cB = \beta \ts
  \id_\cB$. It follows that  $(\beta,\alpha)$ induces an
  isomorphism $\cR \to \cS$.

  It remains to prove \ref{strict}. To check this, consider a commutative diagram:
  \[\xymatrix@-2ex{
    0 \ar[r] &  \cA \ar[d]^-{0} \ar^{i}[r] &  \cD  \ar[d]^-{\lambda} \ar[r]^{p} &  \cB \ar[d]^-{0} \ar[r] & 0\\
    0 \ar[r] &  \cA  \ar[r]^{i'} & \cD' \ar[r]^{p'} & \cB  \ar[r] & 0
  }
  \]

  Since $p'\circ\lambda=0$, we have $\im \lambda\subseteq \cA$. But
  $\lambda i=0$ implies that $\lambda$ factors as:
  \[
  \xymatrix@C-4ex@R-2ex{
    \cD\ar[rrr]^{\lambda} \ar[dr] && &  \cA \\
    & \cB \ar@{=}[r] &\cD/\cA \ar[ur]_-{\bar{\lambda}}  &
  }
  \]
  If $\lambda \ne 0$, this would give a nonzero map $\bar \lambda : \cB \to \cA$, contradicting $\Hom_X(\cB,\cA)=0$.
  With this in mind, we deduce the injectivity of the natural map:
  \begin{equation}
    \label{natural}
  \Hom_\Upsilon(\cR,\cS) \to   \Hom_X(\Phi(\cR),\Phi(\cS)).
  \end{equation}

  As for surjectivity, given a morphism $\mu : \Phi(\cR) \to
  \Phi(\cS)$, we compose $\mu$ on one side with the projection
  $\Phi(\cS) \to \cB^d$, and with the injection $\cA^a \to \Phi(\cR)$
  on the other side. We obtain thus a map $\cA^a \to \cB^d$, which
  must vanish since $\Hom_X(\cA,\cB)=0$.
  We deduce that $\mu$ defines maps $\cA^a \to \cA^b$ and  $\cB^c \to
  \cB^d$, which must be of the form $\beta \ts \id_\cA$ and $\alpha
  \ts \id_\cB$ by the assumption that $\cA$ and $\cB$ are simple.
  The pair $(\beta,\alpha)$ defines a morphism
  $\cR \to \cS$ whose image via \eqref{natural} is $\mu$.
\end{proof}

We deduce a criterion for an ACM
variety being strictly CM-wild.

\begin{corol} \label{corollario}
In the hypothesis of Theorem \ref{strictly wild by extensions}, case
\ref{nonstrict}, resp. case \ref{strict}, we have:
\begin{enumerate}[label=\roman*)]
  \item if $w \ge 3$ and $\cA$ and $\cB$ are ACM, then $X$ is CM-wild, resp. strictly CM-wild;
  \item if moreover $\cA$ and $\cB$ are Ulrich, then $X$ is Ulrich
    wild, resp. strictly Ulrich wild.
  \end{enumerate}
\end{corol}

\begin{proof}
  By construction of the functor $\Phi$ of Theorem \ref{strictly
    wild by extensions}, the sheaf $\Phi(\cR)$ associated with a
  representation $\cR$ of $\Upsilon$ is ACM (respectively, Ulrich) if
  $\cA$ and $\cB$ are
  ACM (respectively, Ulrich).

  The composition of $\Phi$ with the equivalence $\ACM_X \simeq
  \MCM_{R,0}$ gives the statement in case  \ref{strict}. For case
  \ref{nonstrict}, we further compose $\Phi$ with the inclusion
  $\MCM_{R,0} \to \MCM_{R}$. The resulting functor is a representation
  embedding by Remark \ref{well known}. To see this, we denote by $F$ and
  $F'$ the $R$-modules associated with $\Phi(\cR)$ and $\Phi(\cR')$. We claim that an isomorphism
  $F \to F'$ of graded $R$-modules must have degree
  $0$.

  Indeed, given an injective morphism $F \to F'$
  of degree $t$, that is, an injective map $F \to F'(t)$, assuming $t<0$ and letting $A$ and $B$ be the
  $R$-modules associated with $\cA$ and $\cB$, we see that any
  submodule $A$ of $E$
  maps to zero in $A(t)$ by semistability of $\cA$. Also, $A$ maps to
  zero in $B(t)$ because choosing an injection $B(t) \mono B$ we would
  get a map $A \to B$ which is not an isomorphism. So the map $F \to
  F'(t)$ cannot be injective, namely $t \ge 0$. Using the inverse $F'
  \to F$ we see that $t \ge 0$, so finally $t=0$.

  The same argument applies to idempotents of $F$ and shows that any
  non-trivial splitting of $F$ into $R$-modules takes place in
  $\MCM_{R,0}$. This shows that the functor $\Rep_\Upsilon \to \MCM_R$
  is a representation embedding.
\end{proof}

\begin{rmk}
The hypothesis $\Hom_X(\cB,\cA)=0$ in Theorem \ref{strictly wild by extensions}, case ii), is
necessary. If $\Hom_X(\cB,\cA)\ne 0$ we could indeed consider the map:
\[
\phi:\cD\rx{p}\cB\to \cA\rx{i}\cD
\]
This map $\phi$ is not zero, and makes the following diagram commute:
\[
\xymatrix@-2ex{
  0 \ar[r] &  \cA \ar[d]^-{0} \ar^{i}[r] &  \cD  \ar[d]^-{\phi} \ar[r]^{p} &  \cB \ar[d]^-{0} \ar[r] & 0\\
  0 \ar[r] &  \cA  \ar[r]^{i'} & \cD \ar[r]^{p'} & \cB  \ar[r] & 0
}
\]
 Notice, however that any $\phi$ fitting in such a diagram will be
 nilpotent.

As an explicit example, let $X\subset\PP^{m+1}$ be a hypersurface of degree
$d$ and $Z\subset X$ be an arithmetically Gorenstein (which is to say, $\bk[Z]$
is a graded Gorenstein ring) subscheme of codimension two and index
$i_Z$, where:
\[
i_Z=\max \{s \in \ZZ \mid \HH^{m-1}(X,\cI_{Z|X}(s))\neq 0\}.
\]
Define $e=i_Z+m+2-d$ and assume that $e<0$ so that $\Hom_X(\cI_{Z| X}(e),\cO_X)\neq
0$. Let $\cD$ be the sheaf fitting as the middle term of the non-trivial extension of $\cI_{Z|
  X}(e)$ by $\cO_X$. (One can show that this extension exists and is
unique up to a nonzero scalar, by the definition of $e$ and by Serre duality.)

Whenever $Z$ is not a complete intersection inside $X$, $\cD$ is indecomposable.
Anyway $\cD$ is an ACM sheaf of rank $2$ over $X$ which is never simple, as it
always admits a nonzero nilpotent endomorphism. The conclusion of Theorem
\ref{strictly wild by extensions} fails in this case.

\end{rmk}

\section{Stable syzygies of Ulrich modules}

Let $X \subset \PP^n$ be a closed ACM subscheme of dimension $m \ge 1$,
let $Y$ be a general linear section of $X$ of
codimension $c < m$. Set $T=\bk[Y]$, $R=\bk[X]$ and write $\omega_Y$
for the dualizing sheaf of $Y$. The ideal $I_{Y\mid X}$ of $Y$ in $X$ is generated by a regular
sequence of linear forms of $R$ of length $c$.
Looking at a finitely generated graded module $E$ over $T$ as a graded
module over $R$, we take its minimal
graded $R$-free resolution:
\begin{equation}
  \label{reso}
  0 \la E \la F_0 \lx{d_1} F_1 \la \cdots \la F_{\ell-1} \lx{d_\ell} F_{\ell} \la \cdots
\end{equation}
Write $\Om^\ell_R(E)$ for the $\ell^\mathrm{th}$ syzygy of $E$ over
$R$, by which we mean
$\Om^\ell_R(E)=\im(d_\ell)$. It is well-known that, if $E$ is MCM
over $T$, then $\Om_R^\ell(E)$ is MCM over $R$ for $\ell \ge c$.

Let $\SMCM_{R}$ be the stable category of graded maximal
Cohen-Macaulay (MCM) modules over $R$.
Given $E,E'$ in $\MCM_{R}$, we write $\SHom_R(E,E')$ for the morphisms
in this category, namely the morphisms from $E$ to
$E'$, modulo the ideal of morphisms that factor
through a free $R$-module. We write $\SHom_R(E,E')_t$ for the graded
piece of degree $t$ of $\SHom_R(E,E')$. We will also use the notation
$\SMCM_{R,0}$, the stable category where we take $\SHom_R(E,E')_0$ as
set of morphisms from $E$ to $E'$.
We write $\Pi$ for the stabilization functor:
\[
\Pi : \MCM_{R}  \to \SMCM_{R}.
\]
For $\ell \ge c$, we have also the $\ell^\mathrm{th}$ syzygy stable functor:
  \begin{align*}
  \SOm^\ell : \MCM_{T} & \to \SMCM_{R}, \\
  E &\mapsto \Pi\circ\Om_R^\ell(E).
  \end{align*}

The following theorem is the center of this section and will play a major role throughout the rest of the paper.

\begin{THM} \label{stable FF}
   Let $X \subset \PP^n$ be a closed non-degenerate ACM subscheme of dimension $m \ge
   1$. Assume that $X$ is not of minimal degree, namely,
   $\deg(X)>n-m+1$. Then the restriction to $\Ulr_{T,0}$
   of the $c^\mathrm{th}$ stable syzygy functor provides a fully faithful embedding:
  \begin{align*}
  \SOm^c : \Ulr_{T,0}  & \to \SMCM_{R,0}.
  \end{align*}
\end{THM}

We start with a lemma that characterizes varieties of minimal degree
by a negativity condition on the canonical sheaf.

\begin{lem}\label{omega-non-zero}
  Let $X \subset \PP^n$ be a closed ACM subscheme of dimension $m \ge 1$. Then $X$ has minimal degree $\deg(X)=n-m+1$ if and only if  $\HH^0(X,\omega_X(m-1))= 0$.
\end{lem}

\begin{proof}
  Without loss of generality, we may assume that $\bk$ is
  algebraically closed. Let us work by induction on $m$. For $m=1$,
  the statement holds as an ACM
  subscheme $X\subset \PP^n$ has
  minimal degree if and only if the sectional genus of $X$ is zero,
  see the discussion at \S \ref{minimal}. For $m\ge 2$, if $Y$ is a
  hyperplane section of $X$,  the adjunction
  formula gives an exact
  sequence
  \[
  0 \to \omega_{X}(m-2) \to \omega_X(m-1) \to \omega_Y(\dim(Y)-1) \to 0
  \]
  Because $X$ is ACM and $m \ge 2$ we get $\HH^1_*(X,\omega_X)=0$.
  So, since by the induction hypothesis the statement holds for $Y$, taking global sections
  of the above sequence we see that it also holds for $X$.

  For a proof in the language of modules, note that the dual of the
  minimal $S$-resolution of $\bk[X]$
  provides a resolution of the canonical module $K_X$ (see \cite[Remark~1.2.4]{migliore:liaison}) and therefore, by \cite[Theorem~0.4]{eisenbud-green-hulek-popescu}, $X$ is of minimal degree if and
  only if it is $2$-regular if and only if
  $\HH^0(X,\omega_X(m-1))=(K_X)_{m-1}=0$.
\end{proof}

The following lemma will be one of
the keystones of our analysis. Given a finitely generated graded $R$-module $M$, we write $\langle
M_{\le d}\rangle$ for the graded submodule of $M$ generated by the elements of
degree at most $d$ of $M$. We also write $M^*$ for the dual $\Hom_R(M,R)$.

\begin{lem} \label{degree}
  Fix the hypothesis as in Theorem \ref{stable FF}, let $L$ be an Ulrich module over $T$, and set $M=\Om_R^c(L)$.

  Then we have a functorial
  exact sequence:
  \[
  0 \to \langle M^*_{\le {1-c}} \rangle \to M^* \to \Hom_T(L,T(c)) \to 0.
  \]
\end{lem}

\begin{proof}
  Recall that $L$ is generated in a single
  degree, and that the number of minimal generators of $L$ equals $\alpha_0 =
  \deg(X) \rk(L)$. In other words, we can assume $F_0 \simeq R^{\alpha_0}$.
  By the minimality of the resolution, we have, for $i \ge 1$:
  \[
  F_\ell \simeq \bigoplus_{j \ge \ell} R(-j)^{\alpha_{\ell,j}}, \qquad \mbox{for
    some integers $\alpha_{\ell,j}$}.
  \]
  Also, $X$ is not a linear space, so neither is $Y$, so that $L$ has no free
  summands.

We are going to apply the functor  $\Hom_R(-,R)$ to \eqref{reso}. We have:

\begin{itemize}
  \item $\Ext^i_R(L,R)=0$ for $i=0,\dots,c-1$. This is due to the fact
    that the smallest integer $l$ such that $\Ext^l_R(L,R)\neq 0$ is the grade
    (see \cite[Definition 1.2.11 and ff.]{winfried1998cohen}), and
  $$
  \grade(M):=\grade(\mathrm{Ann}_R(L),R)=\grade(I_{Y\mid X},R)=c.
  $$

  \item $\Ext^c_R(L,R)\simeq\Hom_T(L,T(c))$, which follows from noticing that $I_{Y\mid X}$ is generated by a regular sequence of linear forms of length $c$ and applying inductively the graded version of Rees' Theorem (see \cite[Theorem 8.34]{rotman}).
\end{itemize}

 From the previous remarks, we obtain the long exact sequence:
  \begin{equation}
     \label{generates degree -c}
  0 \to F_0^* \to \cdots \to F_{c-2}^* \xrightarrow{d_{c-1}^*} F_{c-1}^* \xrightarrow{\pi} M^* \to \Hom_T(L,T(c)) \to 0.
  \end{equation}

  Now comes the main point, namely that $\Hom_T(L,T)_{1}=0$. To see
  this, recall the isomorphism $\Hom_T(L,T)_{1} \simeq \Hom_Y(\cL,\cO_Y(1))$ and
  that $\cL$ is Ulrich on $Y$ as well as
  $\cHom_Y(\cL,\omega_Y)$. Then, by Serre duality and \cite[Proposition 2.1]{eisenbud-schreyer-weyman} we get:
  \[
  \Hom_Y(\cL,\omega_Y(m-c)) \simeq  \HH^{m-c}(Y,\cL(c-m))^*=0.
  \]
  Note that, by the assumption and Lemma \ref{omega-non-zero}, $\HH^0(Y,\omega_Y(m-c-1)) \ne 0$  and therefore we get
  an embedding:
  \[\cO_Y(1) \mono \omega_Y(m-c).\]
  Therefore,
  $\Hom_Y(\cL,\omega_Y(m-c)) = 0$ implies $\Hom_Y(\cL,\cO_Y(1))=0$.

  We have thus established that
  $\Hom_T(L,T(c))$ contains no element of degree $\le 1-c$. Also, we
  may write:
  \[
  F_{c-1}^* =
  R(c-1)^{\alpha_{c-1,c-1}} \oplus
  R(c)^{\alpha_{c-1,c}}\oplus \cdots
  \]
  Then, \eqref{generates degree -c} says that $F_{c-1}^*$ generates all
  the elements of $M^*$ of degree at most $1-c$, that is,
  the image of $\pi$ is the
  submodule $\langle M^*_{\le {1-c}} \rangle$ of $M^*$.
  This is clearly functorial, and the lemma is  proved.
\end{proof}

\begin{proof}[Proof of Theorem \ref{stable FF}]
Let $L$ and $N$ be two Ulrich modules over $T$. Our goal will be to
describe two mutually inverse maps:
\[
\Hom_T(L,N)_0 \leftrightarrows \SHom_R(\Om^c_R(L),\Om^c_R(N))_0.
\]

Set $M=\Om^c_R(L)$ and $P=\Om^c_R(N)$.
First, let $\varphi : L \to N$ belong to $\Hom_T(L,N)_0$.
Consider the minimal graded free resolutions of $L$ and $N$ over $R$
and choose a lifting of $\varphi$ to these resolutions:
\begin{equation}
  \label{lifting}
\xymatrix@C-3ex{
0  & \ar[l] L \ar^\varphi [d] & \ar[l] F_0 \ar^{\varphi_0}[d] & \ar[l]
\cdots & \ar[l] F_{c-1} \ar^{\varphi_{c-1}}[d] & \ar[l] M
\ar@{.>}^{\tilde \varphi}[d]\\
0  & \ar[l] N & \ar[l] G_0 & \ar[l] \cdots & \ar[l] G_{c-1}  & \ar[l] P
}
\end{equation}

The morphism $\tilde \varphi$ induced on the $c^\mathrm{th}$ syzygy modules
gives the class $\bar \varphi$ in $\SHom_R(M,P)_0$.
This does not depend on the choice of the lifting
$\varphi_i$, as any other choice would provide a map $\tilde \varphi'$
such that $\tilde \varphi - \tilde \varphi'$ factors through a free module.
\medskip

Conversely, given $\bar \psi \in \SHom_R(M,P)_0$, we
choose a representative $\psi : M \to P$ with dual $\psi^* : P^* \to M^*$.
Since $\psi^*$ is homogeneous of degree $0$, it maps elements of
degree at most $1-c$ in $P^*$ to elements of degree at most $1-c$ in $M^*$. By Lemma \ref{degree}
we obtain a diagram:
\begin{equation}
  \label{psi*}
\xymatrix@-2ex{
  0 \ar[r] &  \langle M^*_{\le {1-c}} \rangle \ar[r] &  M^* \ar[r]  &  \Hom_T(L,T(c)) \ar[r] & 0\\
  0 \ar[r] & \langle P^*_{\le {1-c}} \ar_-{\psi^*}[u] \rangle \ar[r] &
  P^* \ar[r] \ar_-{\psi^*}[u]  &  \Hom_T(N,T(c)) \ar_-{\hat \psi}@{.>}[u]\ar[r] & 0
}
\end{equation}

We wish to associate with $\bar \psi$ the morphism $\hat \psi^* : L
\to N$. To do this, we have to check that $\hat \psi$ does not depend
on the choice of the representative $\psi$ of $\bar \psi$. By
definition any other representative differs from $\psi$ by a map
$\zeta : M \to P$ that factors through a free module, which we call
$F$, which means $\zeta=\zeta_2 \zeta_1$ with $\zeta_1 : M \to F$ and
$\zeta_2 : F \to P$. Therefore $\zeta^*$ factors through $F^*$ and
again by Lemma
\ref{degree} we get the
commutative diagram:
\[
\xymatrix@-2ex{
  0 \ar[r] &  \langle M^*_{\le {1-c}} \rangle \ar[r] &  M^* \ar[r]  &  \Hom_T(L,T(c)) \ar[r] & 0\\
  0 \ar[r] & \langle F^*_{\le {1-c}} \rangle \ar_-{\zeta_1^*}[u] \ar[r] &
  F^*\ar_-{\zeta_1^*}[u]  \\
  0 \ar[r] & \langle P^*_{\le {1-c}} \rangle \ar_-{\zeta_2^*}[u]  \ar[r] &
  P^* \ar[r] \ar_-{\zeta_2^*}[u]  &  \Hom_T(N,T(c)) \ar^-{\hat \zeta}@{.>}[uu]\ar[r] & 0
}
\]
Call $G$ the quotient $F^*/\langle F^*_{\le {1-c}} \rangle$. The
diagram says that $\hat \zeta$ factors through $G$.

Now observe that $G$ is a free $R$-module. Indeed, any direct summand of
$F$ takes the form $R(a)$ for some $a \in \ZZ$, and:
\begin{equation}
  \label{free}
\langle R(a)_{\le {1-c}} \rangle = \left\{
  \begin{array}{ll}
    R(a), & \mbox{if $a \ge c-1$,} \\
    0, & \mbox{if $a< c-1$}.
  \end{array}
\right.
\end{equation}

Therefore $G$ is the direct sum of all summands $R(a)$ of $F^*$ with
$a<c-1$, hence $G$ is a graded free $R$-module. But $\Hom_T(N,T(c))$
is a torsion $R$-module.
So it admits no non-trivial morphism with target in $G$, and therefore $\hat \zeta=0$.
\medskip

Let us check now that these maps are mutually inverse.
Given $\varphi \in \Hom_T(L,N)_0$, we consider the representative
$\psi=\tilde \varphi$ of the class $\bar \varphi$. Dualizing
\eqref{lifting} we obtain a commutative diagram:
\[
\xymatrix@-2ex{
  \cdots \ar[r] &  F_{c-1}^* \ar[r] &  M^* \ar[r]  &  \Hom_T(L,T(c)) \ar[r] & 0\\
  \cdots \ar[r] &  G_{c-1}^* \ar_-{\varphi_{c-1}^*}[u] \ar[r] &
  P^* \ar[r] \ar_-{\psi^*}[u]  &  \Hom_T(N,T(c)) \ar_-{\varphi^*}@{.>}[u]\ar[r] & 0
}
\]
This diagram is the extension of \eqref{psi*} to a minimal resolution
of $\langle P^*_{\le {1-c}} \rangle$ and $\langle M^*_{\le {1-c}}
\rangle$.
This says that $\hat \psi = \varphi^*$, so $\hat \psi^* = \varphi$.
\medskip

Conversely, let $\psi$ be a representative of $\bar \psi \in
\SHom_R(M,P)_0$ and set $\varphi=\hat \psi^*$.
Let us lift the map $\langle P^*_{\le {1-c}} \rangle \to \langle M^*_{\le
  {1-c}} \rangle$ induced by $\psi$ to the minimal graded free
resolutions of these modules and dualize to obtain:
\[
\xymatrix@C-3ex{
0  & \ar[l] L \ar^\varphi [d] & \ar[l] F_0 \ar^{\psi_0}[d] & \ar[l]
\cdots & \ar[l] F_{c-1} \ar^{\psi_{c-1}}[d] & \ar[l] M
\ar^{\psi}[d] & \ar[l] 0\\
0  & \ar[l] N & \ar[l] G_0 & \ar[l] \cdots & \ar[l] G_{c-1}  & \ar[l] P&\ar[l] 0
}
\]

Then $\psi$ is induced by a lifting of $\varphi : L \to N$ to the
minimal resolutions of $L$ and $N$. Our proof is thus complete.
\end{proof}

We isolate the following consequence of Lemma \ref{degree}.

\begin{lem} \label{not free}
  If $L$ is an Ulrich module $L$ over $T$, then $\Om^c_R(L)$ has no
  free summands.
\end{lem}

\begin{proof}
  Suppose $\Om^c_R(L)=M\oplus F$, with $F$ a nonzero direct
  summand.
  Lemma \ref{degree} gives:
  \[
  0 \to
    \langle M _{\le 1-c} \rangle  \oplus
    \langle F _{\le 1-c} \rangle
  \xrightarrow{\pi}
    M
    \oplus
    F
  \to
  \Hom_T(L,T(c))
  \to 0,
  \]
  with $\pi$ block-diagonal. Then, \eqref{free} says that the restriction of
  $\pi$ is an
  isomorphism between $\langle F _{\le 1-c} \rangle$ and $F$, as
  $L$ has no free summand.
  Therefore, looking at \eqref{generates degree -c} we see that $d_c$
  is surjective onto $F$, which contradicts minimality of the
  resolution \eqref{reso}.
\end{proof}

\begin{eg}
  Let $X \subset \PP^{m+1}$ be a hypersurface of degree $d$ and $Y$ be
  a linear section of codimension $c$ of $X$. Based on the theory of matrix factorizations developed in \cite{eisenbud:homological},
  the resolution of an Ulrich module $L$ of rank $r$ on $T=\bk[Y]$ reads:
  \begin{equation}
    \label{resL}
  0 \leftarrow L \leftarrow T^{rd} \leftarrow T(-1)^{rd} \leftarrow
  T(-d)^{rd} \leftarrow \cdots
  \end{equation}
  Since $L(-d) \simeq \ker(T(-1)^{rd} \to T^{rd})$, this yields a
  resolution:
  \begin{equation}
    \label{resdualL}
    0 \leftarrow \Hom_T(L,T(d-1)) \leftarrow T^{rd} \leftarrow T(-1)^{rd} \leftarrow
  T(-d)^{rd} \leftarrow \cdots
  \end{equation}

  Combining \eqref{resL} with the Koszul resolution of $Y$ in $X$ we get a
  resolution over $R=\bk[X]$:
  \begin{equation}
    \label{resmix}
    0 \leftarrow L \leftarrow R^{rd} \leftarrow R(-1)^{rd(c+1)} \leftarrow
    R(-2)^{rd(c+\binom{c}{2})} \oplus R(-d)^{rd} \leftarrow \cdots
  \end{equation}
  The $k^\mathrm{th}$ term $F_k$ of this resolution looks as follows (here
  $\varepsilon \in \{0,1\}$):
  \[
  F_k = \bigoplus_{2h+\varepsilon + j = k}
  R(-(j+h d+\varepsilon))^{\binom{c}{j} rd}.
  \]

  Let $M=\Om^c_R(L)$. The resolution of the dualized syzygy
  $M^*$ starts with:
  \[
  \cdots \to R(c-d)^{rd(c+1)} \oplus F_{c-2}^* \to R(c-d+1)^{rd} \oplus F_{c-1}^* \to M^* \to 0.
  \]

  Now we may remove from this resolution the dual of the truncation at
  $M=\Om^c_R(L)$ of \eqref{resmix}, which is to say, by Lemma \ref{degree}, the resolution of $\langle M^*_{1-c} \rangle$.
  The residual strand recovers precisely \eqref{resdualL}, twisted by $R(c-d+1)$.
  The two strands of the resolution do not mix if $d > 2$.
\end{eg}

\begin{rmk}
  Theorem \ref{stable FF} is sharp, in the sense that it fails in
  general for ACM closed schemes
  $X\subset\PP^n$ of minimal degree. Take for instance
  $\cha(\bk)\ne 2$, choose a positive integer $k$ and let
  $X$ be a smooth quadric hypersurface in $\PP^{2k+1}$. Let $Y$ be
  a smooth hyperplane section of $X$.
  It is well-known that, over  $T=\bk[Y]$ there is a unique indecomposable
  ACM (and Ulrich) module $L$ with no free summands, namely the
  module associated with the spinor bundle.
  This module has rank $2^{k-1}$. On the other hand $R=\bk[X]$
  supports exactly two non-isomorphic ACM (and Ulrich) modules
  $F'$ and $F''$, which have both rank $2^{k-1}$.
  There is a short exact
  sequence
  \[
  0\la L \la R^{2^k}\la F'\oplus F''\la 0, \qquad \mbox{hence
    $\Omega_{R}^1(L)\simeq F' \oplus F''$}.
  \]

  Therefore, the functor $\SOm^1$ is not even a representation
  embedding as it sends indecomposable modules to decomposable
  ones. The condition of preserving non-isomorphy of modules also fails in general, as choosing $n=2k+2$ we get $L'$ and $L''$ non-isomorphic
  spinor modules on $T$, but both $\Omega_{R}^1(L')$ and
  $\Omega_{R}^1(L'')$ are isomorphic to the single spinor module on $X$.
\end{rmk}

For ACM
schemes of minimal degree,
even though Theorem \ref{stable FF} cannot be applied, the following proposition shows that the
syzygy functor enjoys a somehow opposite kind of nice feature, namely
it preserves the property of being Ulrich.

\begin{prop}
  Let $X\subset\PP^n$ be a variety of minimal degree $d=\deg(X)=n-m+1$
  and dimension $m\geq 1$. Let $Y\subset X$ be a general linear
  section of codimension  $c<m$. Then, for any Ulrich module $L$ over
  $T=\bk[Y]$, the $c^\mathrm{th}$ syzygy module $\Omega_R^c(L)$ is an Ulrich
  module over $R=\bk[X]$.
\end{prop}
\begin{proof}
  By induction, we can suppose that $c=1$. Therefore, given an Ulrich
  module $L$ over $T$, since we already know that $\Omega_R^1(L)$ is
  MCM over $R$, we just need to show that it is minimally generated by
  $d\rk(\Omega_R^1(L))$ elements. From
  \cite[Prop. 2.1]{eisenbud-schreyer-weyman}, we see that
  the beginning of the minimal resolution of $L$ over $T$ has the
  following form
  \[
  0\la L\la T^{rd}\la T(-1)^{rd(n-m)}\la\cdots
  \
  \]

Therefore, merging it with the minimal resolution of $T$ over $R$ we
obtain that the minimal resolution of $L$ over $R$ starts
  \[
  0\la L\la R^{rd}\stackrel{d_1}{\la} R(-1)^{rd^2}\la\cdots
  \
  \]
 Namely, $\Omega_R^1(L)=\im(d_1)$ is a MCM module over $R$ of
rank $rd$ generated by $rd^2$ elements of the same degree. In other
words, $\Omega_R^1(L)$ is Ulrich.
\end{proof}

\section{CM-wildness from syzygies of Ulrich extensions}

Let us fix the setup for this section. Let $X\subset \PP^n$ be an ACM
subscheme of dimension $m\geq 1$, put $R=\bk[X]$ and let $Y$ be a linear section of
$X$ of codimension $c<m$. Let
$\cA$ and $\cB$ be two simple Ulrich sheaves on
$Y$.
Set $W=\Ext^1_Y(\cB,\cA)$, $w= \dim_{\bk} W$. Here we want to prove
the following fundamental result.

\begin{THM} \label{embedding}
  Assume $w\ge 3$, $X\subset \PP^n$ is not of minimal degree and suppose that $\cA$ and $\cB$ satisfy hypothesis
  \ref{nonstrict} or \ref{strict} of Theorem \ref{itswild}. Then $X$ is of wild CM-type.
\end{THM}

Let us assume for the time being that $w \ne 0$ and write $\Upsilon=\Upsilon_w$. Over $Y \times \PP (W^*)$, there is
a universal extension:
\[
0 \to \cA \boxtimes \cO_{\PP (W^*)} \to \cU \to \cB \boxtimes \cO_{\PP (W^*)}(-1) \to 0.
\]

Take the sheafified minimal graded free resolutions of $\cA$ and $\cB$
as $\cO_X$-modules, pull-back via $p$ to $X \times \PP (W^*)$, and use the mapping
cone construction to build a minimal graded free resolution of $\cU$
over $X \times \PP (W^*)$:
\[
\xymatrix@-2ex{
& 0 & 0 & 0 \\
0 \ar[r] &  \cA \boxtimes \cO_{\PP (W^*)}\ar[u] \ar[r] &  \cU \ar[u]\ar[r] &  \cB\boxtimes \cO_{\PP (W^*)}(-1) \ar[r] \ar[u]&  0 \\
0 \ar[r] &  \cF_0 \ar[u] \boxtimes \cO_{\PP (W^*)} \ar[r] &  \cH_0  \ar^-{d_0}[u] \ar[r] &  \cG_0 \ar[u] \boxtimes \cO_{\PP (W^*)}(-1) \ar[r] &  0 \\
& \vdots \ar[u]& \vdots \ar[u]& \vdots \ar[u]\\
0 \ar[r] &  \cF_{c-1} \ar[u] \boxtimes \cO_{\PP (W^*)} \ar[r] &  \cH_{c-1}  \ar[u] \ar[r] &  \cG_{c-1} \ar[u] \boxtimes \cO_{\PP (W^*)}(-1) \ar[r] &  0 \\
0 \ar[r] &  \cF_{c} \ar[u] \boxtimes \cO_{\PP (W^*)} \ar[r] &  \cH_{c}  \ar^-{d_c}[u] \ar[r] &  \cG_{c} \ar[u] \boxtimes \cO_{\PP (W^*)}(-1) \ar[r] &  0 \\
& \vdots \ar[u]& \vdots \ar[u]& \vdots \ar[u]
}
\]
Here, $\cH_i=\cF_i\boxtimes \cO_{\PP (W^*)} \oplus \cG_i\boxtimes \cO_{\PP (W^*)}(-1)$.
Set:
\[
\cV = \im(d_c).
\]

Then we consider:
\[
\Phi_\cV=\bR p_* (q^*(-) \ts \cV) : \bD(\Coh_{\PP (W^*)}) \to \bD(\Coh_X).
\]

\begin{lem}
Let $T=\bk[Y]$, and set $L$ and $N$ for the modules of global sections of
$\cA$ and $\cB$. Let $\cM$ and $\cN$ be the sheafifications of
$\Om_R^c(L)$ and $\Om_R^c(N)$.
Then:
\[
\Phi_\cV(\cO_{\PP (W^*)}) \simeq \cM, \qquad \Phi_\cV(\Omega_{\PP (W^*)}(1)) \simeq \cN[-1].
\]
\end{lem}

\begin{proof}
By the diagram we have an exact sequence:
\[
0 \to \cM \boxtimes \cO_{\PP(W^*)} \to \cV \to \cN \boxtimes
\cO_{\PP(W^*)}(-1) \to 0,
\]
We get the conclusion  by using Künneth formula, see \cite[\S 3.3]{huybrechts:fourier-mukai},
and the vanishing of cohomology of $\cO_{\PP(W^*)}$ and
$\Omega_{\PP(W^*)}$ except in degree $0$ and $1$, respectively.
\end{proof}

Now consider the equivalence $\Xi$ of \eqref{Xi}.
Then the restriction of $\Phi_\cV \circ \Xi$ to
  $\Rep_\Upsilon$, composed with the global sections functor, gives an exact functor:
\[
\Psi_0 : \Rep_\Upsilon \to  \MCM_{R,0}.
\]
We denote by $\Psi$ the induced functor $\Rep_\Upsilon \to  \MCM_{R}$. Theorem \ref{embedding} amounts to the next result.

\begin{thm}
  If $X\subset \PP^n$ is not of minimal degree and $\cA$ and $\cB$ are
  Ulrich, then $\Psi : \Rep_\Upsilon \to  \MCM_R$ is a
  representation embedding.
  So if $\cA$ and $\cB$ satisfy hypothesis
  \ref{nonstrict} or \ref{strict} of Theorem \ref{itswild}
  and $w \ge 3$, $X$ is of wild CM representation type.
\end{thm}

\begin{proof}
  By construction we have the commutative diagram of functors:
  \[
  \xymatrix{
    \Rep_\Upsilon \ar^-{\Psi_0}[r] \ar^-\Phi[d] & \MCM_{R,0} \ar^-{\Pi}[d] \\
    \Ulr_{T,0} \ar^-{\SOm_R^c}[r] & \SMCM_{R,0}
  }
  \]

  We proved in Theorem \ref{stable FF} that $\SOm^c_R$ is fully
  faithful, and in Theorem \ref{strictly wild by extensions} that
  $\Phi$ is also fully faithful. So the same happens to $\SOm^c_R
  \circ \Phi$ and hence
  to $\Pi \circ \Psi_0$.

  Therefore, if $\cR$ and $\cS$ are two representations of $\Upsilon$
  such that $\Psi_0(\cR) \simeq \Psi_0(\cS)$,  we still have an isomorphism $\Pi (\Psi_0(\cR))
  \simeq \Pi (\Psi_0(\cS))$ and thus $\cR \simeq \cS$ by full
  faithfulness.

  Moreover, if $\Psi_0(\cR)$ is decomposable, then
  $\Hom_R(\Psi_0(\cR),\Psi_0(\cR))_0$ contains a non-trivial idempotent
  $\psi$.
  The class $\bar \psi$ is also an idempotent, which is trivial if and
  only if the summand of $\Psi_0(\cR)$ associated with $\psi$ is
  free. But this cannot happen by Lemma \ref{not free}. Also, again by full
  faithfulness of $\Pi \circ \Psi_0$, $\bar \psi$ corresponds to a
  non-trivial idempotent of $\cR$, so $\cR$ is also decomposable.
  This finishes the proof that $\Psi_0$ is a representation embedding.

  The consequence that $\Psi$ is also a representation embedding
  follows from the argument of Corollary \ref{corollario}.
  Therefore, $X$ of wild CM-type.
\end{proof}

\begin{eg}
  The first class of varieties where it is a priori unknown how to construct large
  families of ACM bundles is given by general cubic
  hypersurfaces of dimension $m \ge 4$. We can do this with Theorem
  \ref{embedding}. Indeed, start with a cubic hypersurface $X$ in
  $\PP^{m+1}$, sufficiently general to admit a smooth surface section
  $Y$. Then we may take $\cA = \cO_Y(A)$ and $\cB=\cO_Y(B)$, where $A$
  and $B$ are twisted cubics in $Y$ meeting at $5$ points,
  see \cite{dani:cubic:ja} : these will satisfy the assumptions of
  Theorem \ref{embedding}. Indeed, $\HH^0(\omega_Y(m-c-1))=\HH^0(\cO_Y)=\bk$ and, by Riemann-Roch:
  \[
  \dim_\bk\Ext^1(\cB,\cA)=-\chi(\cO_Y(A-B))=3
  \]
  In the next section we will see how to deal in a similar fashion
  with any variety besides the non-wild varieties listed in the main result.
\end{eg}

\section{Wildness of non-integral projective schemes}

In this short section we pay attention to the case of non-integral projective schemes. We first focus on the case of reducible
subschemes. For them, the next lemma is our starting point. In order to use it one starts with an ACM subscheme
$X$, recalls that $X$ is thus equidimensional, takes
$X_0$ to be a union of components of $X$, and studies the representation
type of $X$ in terms of that of $X_0$. Let us put $R=\bk[X]$ and $R_0=\bk[X_0]$.

\begin{lem} \label{one-piece}
  Let $X_0 \subset X\subset\PP^n$ be closed subschemes of the same
  dimension $m$
  and suppose that $X_0$ is CM-wild. Then $X$ is CM-wild.
\end{lem}
\begin{proof}
  The inclusion $X_0\subset X$ gives a surjective morphism of
  rings $R \to R_0$ that bestows a
  structure of $R$-module to any $R_0$-module. Because $X_0$
  and $X$ have the same dimension,
  any MCM $R_0$-module is
  also an MCM $R$-module. Non-isomorphic
  $R_0$-modules remain non-isomorphic $R$-modules. Also,
  an indecomposable $R_0$-module is indecomposable as
  $R$-module. In other words, we have a representation embedding
  $\MCM_{R_0} \to \MCM_R$, so the lemma is proved.
\end{proof}

As a consequence of the previous lemma, in order to classify
reducible projective schemes, it only remains to take care of
reducible varieties having no CM-wild component. This is the content
of the following result.

\begin{thm} \label{two-pieces}
  Let $X_1, X_2  \subset \PP^n$ be $m$-dimensional closed integral
  varieties with  $m \ge 2$. Assume that $X_1\cap X_2$ is a Weil divisor in
  $X_1$, that $X_2$ is ACM and that $X_1$ carries an ACM sheaf.
  Then $X_1\cup X_2$ is CM-wild.
\end{thm}
\begin{proof}
  Put $X=X_1 \cup X_2$ and $Y = X_1 \cap X_2$. The surjection $\cO_X
  \to \cO_Y$ factors as $\cO_X \to \cO_{X_1} \to \cO_Y$ and  $\cO_X
  \to \cO_{X_2} \to \cO_Y$ and, since $X=X_1 \cup X_2$, this induces
  an isomorphism of $\cO_X$-sheaves $\cI_{X_2|X} \simeq \cI_{Y|X_1}$.
  In other words, we have an exact sequence:
  \[
  0 \to \cI_{Y|X_1} \to \cO_X \to \cO_{X_2} \to 0.
  \]

  Let $\cF_1$ be an ACM sheaf on $X_1$. Note that
  $\cHom_X(\cO_{X_2},\cF_1)=\cHom_X(\cF_1,\cO_{X_2})=0$, because
  $\cF_1$ and $\cO_{X_2}$ are respectively supported on $X_1$ and $X_2$, and these
  varieties have no
  common component.
  So:
  \[\Hom_X(\cO_{X_2},\cF_1(q))=\Hom_X(\cF_1(q),\cO_{X_2})=0, \qquad
  \mbox{for all $q \in \ZZ$}.
  \]
  Also,
  $\cHom_X(\cO_{X},\cF_1) \simeq \cF_1$. So, applying
  $\cHom_X(-,\cF_1)$ to the previous exact sequence,  we get:
  \begin{equation}
    \label{Ext-OX2}
  0 \to \cF_1 \to \cHom_X(\cI_{Y|X_1},\cF_1) \to \cExt^1_X(\cO_{X_2},\cF_1).
  \end{equation}

  We look at $\cI_{Y|X_1}$ as the kernel of $\cO_{X_1} \to \cO_Y$.
  Applying $\cHom_{X_1}(-,\cF_1)$ gives:
  \begin{equation}
    \label{Ext-F1}
  0 \to \cF_1 \to \cHom_{X_1}(\cI_{Y|X_1},\cF_1) \to
  \cExt^1_{X_1}(\cI_{Y|X_1},\cF_1) \to 0,
  \end{equation}
  because $\cF_1$ is locally Cohen-Macaulay. On the other hand,
  applying $\cHom_{X_1}(-,\cO_{X_1})$ gives:
  \[
  0 \to \cO_{X_1} \to \cHom_{X_1}(\cI_{Y|X_1},\cO_{X_1}) \to \cN_{Y|X_1}
   \to 0,
  \]
  where we used the standard identification of the sheaf $\cExt^1_{X_1}(\cI_{Y|X_1},\cO_{X_1})$
  with the normal sheaf $\cN_{Y|X_1}$ of $Y$ in $X_1$.
  Tensoring the previous exact sequence with $\cF_1$ gives:
  \begin{equation}
    \label{normal-F1}
  \cdots \to \cF_1 \to \cHom_{X_1}(\cI_{Y|X_1},\cO_{X_1}) \otimes \cF_1 \to
  \cN_{Y|X_1} \otimes \cF_1 \to 0.
  \end{equation}

  Because $\cHom_{X}(\cI_{Y|X_1},\cF_1) \simeq
  \cHom_{X_1}(\cI_{Y|X_1},\cF_1)$, putting
  together  \eqref{Ext-OX2} and \eqref{Ext-F1} yields an inclusion $\cExt^1_{X_1}(\cI_{Y|X_1},\cF_1) \mono
  \cExt^1_X(\cO_{X_2},\cF_1)$.
  Therefore, for $q \in \NN$, we have a linear inclusion:
  \begin{equation}
    \label{growing}
  \HH^0(X_1,\cExt^1_{X_1}(\cI_{Y|X_1},\cF_1(q))) \mono \HH^0(X,\cExt^1_X(\cO_{X_2},\cF_1(q)) \simeq \Ext^1_X(\cO_{X_2},\cF_1(q)).
  \end{equation}
  where the isomorphism follows from the (degenerate)
  local-to-global spectral sequence.

  There is a dense open subset of the reduced structure over $Y$ where $Y$
  is Cartier and $\cF_1$ is locally free.
  Over such open set, the exact sequences \eqref{Ext-F1} and
  \eqref{normal-F1} become the same and $\cN_{Y|X_1} \otimes \cF_1$ is
  locally free of positive rank. Therefore,
  since $\dim(Y)=m-1 \ge 1$, the dimension of
  $\HH^0(Y,\cN_{Y|X_1} \otimes \cF_1(q))$ grows at least linearly when
  $q \gg 0$. Hence the same happens to $\HH^0(X_1,\cExt^1_{X_1}(\cI_{Y|X_1},\cF_1(q)))$.
  Therefore, in view of \eqref{growing}, $\Ext^1_X(\cO_{X_2},\cF_1(q))$ has unbounded dimension
  for growing $q$.
  The result now follows by applying item \ref{strict} of Theorem \ref{strictly wild by
  extensions} with $\cB=\cO_{X_2}$ and $\cA=\cF_1(q)$ for $q \gg 0$.
\end{proof}

We finish this section with a foray into non-reduced schemes.

\begin{thm} \label{non-reduced}
  Let $X \subset \PP^n$ be an $m$-dimensional closed subscheme containing a double
  structure over an integral ACM subscheme $X_0$ of dimension $m \ge 1$. Then $X$ is CM-wild.
\end{thm}

\begin{proof}
  This follows the same path as the previous lemma. We have an exact sequence:
  \[
  0 \to \cI_{X_0|X} \to \cO_X \to \cO_{X_0} \to 0.
  \]
  Since $X$ contains a double structure over $X_0$, the sheaf
  $\cI_{X_0|X}$ has rank at least one as a sheaf over $X_0$.
  Applying $\cHom_X(-,\cO_{X_0})$ to this sequence, we get an exact
  sequence:
  \[
  0 \to \cO_{X_0} \xrightarrow{\simeq} \cO_{X_0} \to
  \cHom_X(\cI_{X_0|X},\cO_{X_0}) \to \cExt^1_X(\cO_{X_0},\cO_{X_0}).
  \]
  Since the ideal sheaf $\cI_{X_0|X}$ has rank at least one, the sheaf $\cE
  = \cHom_X(\cI_{X_0|X},\cO_{X_0})$ is (torsion-free) of rank at least
  one over $X_0$ as well.
  Therefore, since $m\ge 1$, for $q \gg 0$, the dimension of $\HH^0(X,\cE(q))$ is
  unbounded, and thus also the dimension of $\HH^0(X,\cExt^1_X(\cO_{X_0},\cO_{X_0}(q)))$.

  We use now the exact sequence of lower
  degree terms of the local-to-global spectral sequence, together with
  the fact that $\cO_{X_0} \simeq \cHom_X(\cO_{X_0},\cO_{X_0})$. This
  gives an exact sequence:
  \[
  \Ext^1_X(\cO_{X_0},\cO_{X_0}(q)) \to
  \HH^0(X,\cExt^1_X(\cO_{X_0},\cO_{X_0}(q)) \to \HH^2(X,\cO_{X_0}(q))
  \]
  Now by Serre's vanishing $\HH^2(X,\cO_{X_0}(q))=0$ for $q \gg 0$, so
  the dimension of $\Ext^1_X(\cO_{X_0},\cO_{X_0}(q))$ is unbounded for
  $q\gg 0$. The conclusion again follows from item \ref{nonstrict} of Theorem \ref{strictly wild by
  extensions}, applied to $\cB=\cO_{X_0}$ and $\cA=\cO_{X_0}(q)$ with $q \gg 0$.
\end{proof}

\section{Varieties of minimal degree}
\label{minimal}

Assume $\bk$ is algebraically closed and $\cha(\bk)\ne 2$. Let $X \subset
\PP^n$ be a reduced closed ACM subscheme of dimension $m \ge 1$ and
degree $d$.
The subscheme $X$ is thus linearly normal and, without loss of generality, we may assume throughout that
$X$ spans $\PP^n$, so that $X$ is embedded by the complete linear
series of a very ample line bundle $\cO_X(1)$.
We argue on the {\it sectional genus} $p$ of
$X$, that is, the arithmetic genus of a reduced $1$-dimensional linear
section of $X$.
We also introduce the {\it $\Delta$-genus} of $X$, defined as
$\Delta(X)=d-n+m-1$.
Since $X$ is connected in codimension one (see \cite{hartshorne:connectedness} for
the definition and the result), a theorem of Xamb\'o
(see \cite{xambo:minimal}) generalizing the classical lower
bound on $d$
(see for example \cite[Corollary 5.13]{mumford:ag-I}), asserts that
 $\Delta(X)\geq 0$, or in other words $d \ge n-m+1$.
If equality is attained, $X$ is said to be of \textit{minimal
  degree}. This happens if and only if $\Delta(X)=0$, and also if and
only if $p=0$.

Before proceeding, a few words for the non-reduced case are in order: if $X$ is ACM but not
reduced, of dimension $m \ge 1$ and of degree $n-m+1$,
it still makes sense to ask about the representation type of $X$. In
this case, $X$ is a $2$-regular scheme (see {\cite{eisenbud-green-hulek-popescu} and references therein for this notion and related results) and therefore $X$ has a non-reduced irreducible component whose reduction $X_0$ is integral of minimal degree (see \cite[Corollary
0.8 and Theorem 0.4]{eisenbud-green-hulek-popescu}); therefore $X_0$ is ACM, in which case $X$ is CM-wild
by Theorem \ref{non-reduced}.

Let us return to the reduced case:  we can suppose $m \ge 2$, which is harmless since the representation type of curves is well-known,
see \cite{drozd-greuel,drozd-greuel:CM-type,bodnarchuk-burban-drozd-greuel}.
Here is the main result of this section.

\begin{thm} \label{main-minimal}
  Let $X$ be a non-degenerate ACM closed subscheme of minimal degree and dimension
  $m\ge 2$. Then $X$ is
  CM-wild, except if $X$ a linear space, a quadric hypersurface of corank at most $1$,
  the Veronese surface in $\PP^5$, or a smooth rational normal surface scroll of degree $3$ or $4$.
\end{thm}

By \cite[Theorem 0.4 and
1.4]{eisenbud-green-hulek-popescu} and \cite[Theorem
1]{xambo:minimal}, if the scheme $X$ is reducible, then it has at least two reduced irreducible components, both of
minimal degree, meeting along a divisor, so that $X$ is
CM-wild by Theorem \ref{two-pieces}.

Therefore it only remains to see what happens when $X$ is integral. In this case $X$ fits in the classification of del Pezzo and
Bertini, see \cite{eisenbud-harris:minimal-degree}.
 If $X$ is smooth,
then $X$ is either a quadric, or a Veronese surface, or a rational
normal scroll, and the representation type of all these varieties is
known.
Indeed, $X$ is of finite CM-type (see \cite{auslander-reiten:almost-split,eisenbud-herzog:CM}) if it is a linear space (see \cite{horrocks:punctured}), or a
smooth quadric (see \cite{knorrer:ACM}), or the Veronese surface
in $\PP^5$, or a smooth cubic scroll in
$\PP^4$ (see \cite{auslander-reiten}, see also \cite{faenzi:iyama-yoshino}), or a rational normal curve.
Also, $X$ is of tame CM-type if it is a
rational normal surface scroll of degree $4$, see \cite{faenzi-malaspina}.
Besides these cases, $X$ is strictly Ulrich wild, as we see by applying
Theorem \ref{strictly wild by extensions} to the Ulrich line bundles
considered in \cite{miro_roig:scrolls,faenzi-malaspina}. For the
reader's convenience (and because these sheaves will play a role
further on), we recall that, if the scroll $X  \subset \PP^n$
has degree $d$, the Ulrich line bundles are the ideal
of a fibre of the scroll twisted by
$\cO_{X}(1)$, and the dual of the ideal of $d-1$ fibres.

It remains to understand what happens when $X$ is an integral but singular scheme of minimal degree.
The goal of the rest of this section is to settle this point.

\begin{thm} \label{minimal degree}
  An integral, non-degenerate, singular variety $X$ of
  minimal degree and dimension $m \ge 2$ is of wild CM
  type unless $X$ is a quadric of corank $1$.
\end{thm}

According to del Pezzo and Bertini,
singular varieties of minimal degree are cones over smooth
varieties of minimal degree, so we start by studying in some detail the behavior of sheaves defined on cones.

\subsection{Extension of sheaves over cones} \label{extensions-cones}

Let $X \subset \PP^n$ be a closed non-degenerate subscheme. Fix a
linear subspace $\Lambda \subset \PP^n$. Let $\Lambda^0 \subset \PP^n$
be a linear subspace disjoint from $\Lambda$ such that $\Lambda$ and
$\Lambda^0$ span $\PP^n$, and consider a subscheme $X^0 \subset \Lambda^0$.
\begin{dfn}
  We say that $X$ is a \textit{cone} with vertex (or apex) $\Lambda$
  and base $X^0$ if $X$ is the union of all lines joining a point of
  $\Lambda$ and a point of $X^0$.
\end{dfn}

When $X$ is a cone with vertex $\Lambda \subset \PP^n$ of codimension $n_0+1$, any subspace
$\Lambda^0$ disjoint from $\Lambda$ and of dimension $n_0$ provides $X^0=X\cap \Lambda^0$ as base of
$X$.
To write the equations of a cone, choose coordinates so that $\Lambda^0$ is defined by the vanishing of the linear forms
$x_{n_0+1},\ldots,x_n$. We denote $\lambda_i=x_{n_0+1+i}$ the
coordinates of $\Lambda$ so that $\bk[\Lambda] \simeq
\bk[\lambda_0,\ldots,\lambda_{n-n_0-1}]$. Setting $S^0=\bk[x_0,\ldots,x_{n_0}]$,
the ideals $I_{X|S}$ of $X$ in $S$ and $I_{X^0|S^0}$ of $X^0$ in $S^0$
are generated by the same
minimal set of polynomials. Put $R=\bk[X]$ and $R^0=\bk[X^0]$. In terms of graded rings:
\[
R \simeq R^0 \otimes_{\bk} \bk[\Lambda].
\]

\begin{lem} \label{ext-i-tensor}
Given finitely generated $R^0$-modules $E^0$ and $F^0$, set $E =
E^0 \otimes_{R^0} R$ and $F =F^0 \otimes_{R^0} R$. Then, for all $i \ge 0$, we have an isomorphism of graded
$\bk[X]$-modules:
\[
\Ext^i_R(E,F) \simeq \Ext^i_{R^0}(E^0,F^0) \otimes \bk[\Lambda].
\]
\end{lem}

\begin{proof}
We have
\[
\Ext^i_R(E,F)\simeq\Ext^i_{R^0}(E^0,F^0 \otimes_{R^0} R)\simeq\Ext^i_{R^0}(E^0,F^0) \otimes_{R^0}\bk[X]
\]
 where the first isomorphism is \cite[Theorem 11.65]{rotman}, using
 that $R$ is a
  flat $R^0$-module. In order to finish, we need only to observe the standard isomorphism:
\[
\begin{split}
    \Ext^i_{R^0}(E^0,F^0) \otimes_{R^0}(R^0 \otimes_{\bk}
\bk[\Lambda]) & \simeq (\Ext^i_{R^0}(E^0,F^0) \otimes_{R^0} R^0) \otimes_{\bk}\bk[\Lambda] \\
     & \simeq\Ext^i_{R^0}(E^0,F^0) \otimes_{\bk} \bk[\Lambda].
\end{split}
\]
\end{proof}

\subsection{The cubic cone}

We focus now on the only case where all construction methods of
representation embeddings seen so far fail, for two reasons.
First, no extension group grows enough to use Theorem \ref{itswild}. Second,
all the ACM sheaves that do have deformations have
many endomorphisms, which is another obstruction to use Theorem \ref{itswild}.
We develop a specific technique to study this very intriguing case.
For this subsection, the field $\bk$ is arbitrary.

\begin{thm} \label{funnycase}
  The cone $X \subset \PP^4$ over a rational normal cubic curve
  in $\PP^3$ is CM-wild.
\end{thm}

\begin{proof}
  We divide the proof into eight steps.
  \begin{step} Define the sheaf $\cF^0$ on the twisted cubic and
    compute its self-extensions.
  \end{step}

  Let us write $X^0$ for the base of the cone $X$. Using the convention of \S
  \ref{extensions-cones} let us put $R=\bk[X]$, $R^0=\bk[X^0]$ so that $R \simeq \bR^0 \otimes
  \bk[\lambda_0]$, and let us abbreviate $\lambda=\lambda_0$.
  Define $\cF^0$ to be the line bundle of degree
  $2$ on $X^0$, that is,
  $\cF^0 \simeq \cO_{\PP^1}(2)$. This is a stable Ulrich sheaf on
  $X^0$. We write $F^0$ for its associated $\bk[X^0]$-module.
  Set $W$ for the 2-dimensional vector space
  $W=\Ext^1_{X^0}(\cF^0(1),\cF^0) \simeq \HH^1(\PP^1,\cO_{\PP^1}(-3))$.
  By Lemma \ref{Ext-Ext} with $\cE^0=\cF^0(s)$ for all $s \in \ZZ$ we get a graded isomorphism:
  \[
  \Ext^1_{R^0}(F^0(s)_{\ge 0},F^0)_{\ge 0} \simeq \bigoplus_{t
    \ge 0} \Ext^1_{X^0}(\cF^0(s),\cF^0(t)) \simeq \bigoplus_{0 \le t \le s-1} S^{3(s-t)-2}W,
  \]
  where by convention a symmetric power with negative exponent is zero.

  \begin{step} Define the sheaf $\cF$ on $X$ and
    compute its self-extensions.
  \end{step}

  We now pay attention to $X$.   Let $\cF$ be the ideal sheaf of a ray of
  the cone $X$, twisted by $\cO_X(1)$;
  this is a stable Ulrich sheaf of rank $1$ on $X$. The $R$-module of global sections $F$ associated with $\cF$ satisfies $F
  \simeq F^0 \otimes \bk[\lambda]$.
  By Lemma \ref{ext-i-tensor} we have, for all $q\in \ZZ$:
  \[
  \Ext^1_R(F(1),F)_q \simeq \Ext^1_{R^0}(F^0(1),F^0)
  \otimes \bk[\lambda]_q.
  \]
  Using Lemma
  \ref{Ext-Ext} and setting $\lambda^q=0$  by convention for $q<0$, we conclude, for all $q \in \ZZ$:
  \[
  \Ext^1_{X}(\cF(1),\cF(q)) \simeq  \Ext^1_R(F(1),F)_{q} \simeq
  W \cdot \lambda^q.
  \]

  \begin{step} \label{step3} Define the quiver $\Theta$ and a functor
    $\Rep_\Theta \to \ACM_X$.
  \end{step}

  We introduce now the quiver $\Theta$, which we depict as follows.
  
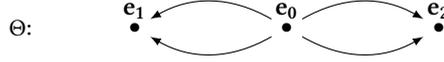
\[
\begin{tikzpicture}[scale=1]
  \draw (-4.5,0) node [] {$\Theta$:};
  \draw (-3,0) node [above] {$ \mathbf{e_1}$};
  \draw (-1,0) node [above] {$ \mathbf{e_0}$};
  \draw (1,0) node [above] {$\mathbf{e_2}$};
  \node (1) at (-3,0) {$\bullet$};
  \node (0) at (-1,0) {$\bullet$};
  \node (2) at (1,0) {$\bullet$};
  \draw[->,>=latex] (0) to[out = -30,in=-150] (2);
  \draw[->,>=latex] (0) to[out = 30,in=150] (2) ;
  \draw[->,>=latex] (0) to[out=150,in=30] (1) ;
  \draw[->,>=latex] (0) to[out=-150,in=-30] (1) ;
    \end{tikzpicture}
  \]

  Let $\cR$ be a
  representation of $\Theta$ with dimension vector $(a_0,a_1,a_2)$, so
  that $\cR$ consists of two pencils of linear maps $A_1
  \leftleftarrows A_0  \rightrightarrows A_2$, where $A_i$ is a vector
  space of dimension $a_i$. As we did in \S
  \ref{kronecker-quiver} we associate a linear
   map with each of these two pencils by indexing the arrows of
   $\Theta$ with basis elements of $W$. This way, the datum or $\cR$
   is tantamount to:
   \[
   (\xi_1,\xi_2) \in \Hom_{\bk}(A_0,A_1) \otimes W \cdot \lambda^2 \times
   \Hom_{\bk}(A_0,A_2) \otimes W \cdot \lambda^3.
   \]
   In other words, we identify $\cR$ with:
   \[
   (\xi_1,\xi_2) \in \Ext^1_{X}(A_0 \otimes \cF, (A_1 \otimes \cF(1)) \oplus (A_2 \otimes
   \cF(2))).
   \]
   By the procedure described in \S \ref{extensions}, this gives
   a sheaf $\cE$ fitting into the exact sequence:
   \[0  \to
   \begin{array}[h]{c}
     A_1 \otimes \cF(1) \\
     \oplus \\
     A_2 \otimes \cF(2)
   \end{array}
   \to \cE \to A_0 \otimes \cF \to 0.
   \]

   The sheaf $\cE$ is thus clearly ACM and has a Jordan-Hölder filtration whose
   associated graded object is $\cF^{a_0} \oplus \cF(1)^{a_1} \oplus
   \cF(2)^{a_2}$. The procedure of constructing $\cE$ from $\xi$, or
   equivalently from $\cR$, is functorial, which can be seen readily
   following the proof of Theorem \ref{strictly wild by
     extensions}.

  \begin{step}
    Define a functor $\Phi : \FMod_{\bk[x_1,x_2]} \to \ACM_X$
    factoring through $\Rep_\Theta \to \ACM_X$.
  \end{step}

  Put $\Sigma=\bk[x_1,x_2]$. We define a functor from
  $\FMod_\Sigma$ towards the category of ACM sheaves over $X$.
  Choose a basis $(w_0,w_1)$ of $W$.
  A finite-dimensional $\Sigma$-module
  is a finite dimensional vector space $M$ together with two
  commuting endomorphisms $\sx_1$ and $\sx_2$.
  We define two
  representations $\xi_1$ and $\xi_2$ by setting
  $A_0=M \oplus M$, $A_1=A_2=M$ and:
  \[
  \xi_1 = (\id_M w_0 + \sx_1 w_1,0), \qquad  \xi_2 = (0,\id_M w_0 + \sx_2 w_1).
  \]
  The sheaf $\cE$ associated with $M$ is defined by the pair
  $\xi=(\xi_1,\xi_2)$ as in Step \ref{step3}.

  \begin{step}
    Prove that, if $\cE=\Phi(M)$ and $\cE'=\Phi(M')$ are
    isomorphic, then
    $M \simeq M'$ as $\bk[x_1]$-modules.
  \end{step}

  Given two finite-dimensional $\Sigma$-modules $M$ and
  $M'$, we have two sheaves $\cE$ and
  $\cE'$. Assume that these sheaves are isomorphic.
  By
  the Harder-Narasimhan filtrations of $\cE \simeq \cE'$ we get that $\cF(2)^{a_2}$ is a
  maximal destabilizing subsheaf of both $\cE$ and $\cE'$, which
  implies that the quotient sheaves $\cE_1=\cE/\cF(2)^{a_2}$ and
  $\cE_1'=\cE'/\cF(2)^{a_2}$ are isomorphic. These sheaves are given
  by the extension classes $\xi_1$ and $\xi_1'$, which are thus
  isomorphic by the argument we used in the proof of Theorem \ref{itswild}.

  We identify $M$ and $M'$ as vector spaces and consider $\sx_1$ and
  $\sx_1'$ as endomorphisms of $M$. Recall the
  expressions $\xi_1=(\id_M w_0 + \sx_1 w_1,0)$ and $\xi'_1=(\id_M w_0 + \sx'_1 w_1,0)$.
  Because $\xi_1$ and $\xi'_1$ are isomorphic, there are linear isomorphisms $\alpha_0
  \in \End_\bk(M\oplus M)$, $\alpha_1 \in \End_\bk(M)$ such
  that $\alpha_1 \otimes
  \id_W \circ \xi_1' = \xi_1 \circ \alpha_0$.
  Decomposing $\alpha_0$ as a block matrix of endomorphisms $\alpha_0^{i,j}$ of $M$, we
  rewrite this as
  \[
  \alpha_1 (\id_M w_0 + \sx_1' w_1,0)= (\id_M w_0 + \sx_1 w_1,0) \begin{pmatrix}
    \alpha_0^{1,1} &     \alpha_0^{1,2} \\
    \alpha_0^{2,1} &     \alpha_0^{2,2}
  \end{pmatrix}
  \]
  In particular we get:
  \[
  \alpha_1 w_0 + \alpha_1 \sx_1' w_1 = \alpha_0^{1,1} w_0 + \sx_1
  \alpha_0^{1,1} w_1,
  \]
  so $\alpha_1=\alpha_0^{1,1}$ and  $\alpha_1$ conjugates
  ${\sf x}_1'$ to  ${\sf x}_1$. Then $M$ and $M'$ are isomorphic as
  $\bk[x_1]$-modules.

  \begin{step}
     Prove that, if the sheaves $\cE =\Phi(M)$ and $\cE' = \Phi(M')$ are
     isomorphic, then  $M \simeq M'$.
  \end{step}

  Suppose again $\cE \simeq \cE'$ and
  assume now $\xi_1=\xi_1'$, which can be achieved after linear
  automorphisms by the previous point.

  By
  definition,  $\cE$
  and $\cE'$ are obtained from $\cE_1=\cE_{\xi_1}$ by
  using $\xi_2$ and $\xi_2'$, which are both linear maps $A_0
  \to A_2 \otimes W$. Let us look more closely at how this is achieved.
  Start with the extension $\xi_1$ and the sheaf $\cE_1$ fitting into:
  \[
  (\xi_1) \qquad 0 \to A_1 \otimes \cF(1) \to \cE_1 \to A_0 \otimes
  \cF \to 0.
  \]
  Apply to this the functor $\Hom_X(-,A_2\otimes \cF(2))$. We get:
  \begin{equation}
    \label{describeExt}
  \Hom_X(A_1\otimes \cF(1),A_2\otimes \cF(2)) \to
  \Ext^1_X(A_0 \otimes \cF,A_2\otimes \cF(2)) \to
  \Ext^1_X(\cE_1, A_2\otimes \cF(2)).
  \end{equation}
  This is rewritten in the form:
  \begin{equation}
    \label{describeExt2}
  \Hom_{\bk}(A_1,A_2) \cdot \lambda \to \Hom_{\bk}(A_0,A_1) \otimes W
  \cdot \lambda^3 \to   \Ext^1_X(\cE_1, A_2\otimes \cF(2)).
  \end{equation}
  Given a linear map $\gamma : A_1 \to A_2$, the leftmost map in the
  previous diagram sends $\gamma \cdot \lambda$ to $\gamma \otimes
  \id_W \circ \xi_1 \cdot \lambda^3$.
  The isomorphism class of the sheaf $\cE$ determines an element of
  $\Ext_X^1(\cE_1,A_2
  \otimes \cF(2))$ coming from a pair $(\xi_1,\xi_2)$, so the previous
  diagram shows that $\cE$ determines the isomorphism class of $\xi_2$
  up to adding any map of the form $\gamma \otimes
  \id_W \circ \xi_1$, for $\gamma \in \Hom_{\bk}(A_1,A_2)$.

  This says that an isomorphism $\cE' \to \cE$ exists if and only if
  there exist a linear isomorphism $\alpha_2\in \End(A_2)$ and a linear
  map $\gamma : A_1 \to A_2$ such that $\xi_2$ is carried to $\xi_2'$
  by $\alpha_2 + \gamma \otimes \id_W$, after composition with
  $\xi_1$, that is:
  \begin{equation}
    \label{repres}
  \alpha_2 \otimes \id_W \circ \xi_2' = \xi_2 + \gamma \otimes \id_W
  \circ \xi_1.    
  \end{equation}

The sheaves
  $\cE$ and $\cE'$  determine isomorphic
  extensions and therefore equal elements in the group $\Ext_X^1(\cE_1,A_2
  \otimes \cF(2))$. This implies that there exist a
  linear isomorphism $\alpha_2 \in \End_\bk(M)$ and a linear map
  $\gamma : A_1 \to 
  A_2$ such that \eqref{repres} holds.
  According to the decomposition $A_0=M \oplus M$, we rewrite this as:
  \[
  (0,\alpha_2 w_0 + \alpha_2 \sx_2'w_1)=(0,\id_M w_0 +
  \sx_2w_1)+(\gamma w_0+\gamma \sx_1 w_1,0).
  \]
  In particular we obtain $\alpha_2=\id_M$ and $\sx_2'=\sx_2$. The modules $M$
  and $M'$ are thus isomorphic.

  \begin{step}
    Prove that, if $\cE=\Phi(M)$ is decomposable, then $M$
    is decomposable as $\bk[x_1]$-module.
  \end{step}
    We
  write $\cE = \cE' \oplus \cE''$. The Harder-Narasimhan filtration of
  $\cE'$ must be compatible with that of $\cE$ and therefore its
  associated graded object must be $\cF^{a'_0} \oplus \cF(1)^{a'_1} \oplus
   \cF(2)^{a'_2}$, and similarly for $\cE''$.
   Then, the quotient $\cE_1'=\cE'/\cF(2)^{a'_2}$ has a graded object
   of the form $\cF(1)^{a'_1} \oplus \cF(2)^{a'_2}$ and therefore,
   as in the proof of Theorem \ref{itswild}, we must have $\cE_1' \simeq
   \cE_{\xi_1'}$ for some $\xi_1' \in \Hom_\bk(A_0',A_1') \otimes W$,
   where $A'_0$ and $A'_1$ are vector spaces of dimension $a_0$ and
   $a_1$ appearing in the vector space decompositions $A_0=A_0'\oplus A_0''$ and
   $A_1=A_1'\oplus A_1''$.
   Likewise we have $\cE_1'' \simeq
   \cE_{\xi_1''}$ for some $\xi_1'' \in \Hom_\bk(A_0'',A_1'') \otimes
   W$ and $\cE_1 = \cE_1' \oplus \cE_1''$, which implies that $\xi_1$
   has a block-diagonal form in terms of $\xi_1'$ and $\xi_1''$.

   Now we have $M=A_1$ decomposed as vector space as $A_1' \oplus
   A_1''$. We put $M'=A_1'$, $M''=A_1''$ and we write $M=M'\oplus
   M''$, so we decompose $A_0=M \oplus M$ as $A_0=M' \oplus M'' \oplus
   M' \oplus M''$.
   By definition we have $\xi_1 = (\id_M w_0+\sx_1 w_1,0)$ so the expression
   $\id_M w_0+\sx_1 w_1$ gives a map $M' \oplus M'' \to (M' \oplus
   M'')\otimes W$ which, in view of the decomposition of $\xi_1$ in
   diagonal form in terms of $\xi_1'$ and $\xi_1''$, takes the form:
   \[
   \id_M w_0+\sx_1 w_1 =
   \begin{pmatrix}
     \id_{M'}w_0 +\sx'_1w_1  &  0 \\
     0 & \id_{M''}w_0 +\sx''_1w_1
   \end{pmatrix},
   \]
   for some linear maps $\sx'_1 : M' \to M'$ and $\sx''_1 : M'' \to
   M''$.

   \begin{step}
     Prove that, if $\cE=\Phi(M)$ is decomposable, then $M$
    is decomposable.
   \end{step}

   We proved that $M=M' \oplus M''$ as $\bk[x_1]$-module. Now we have
   to use $\xi_2$ to prove that the splitting  $\cE =
   \cE' \oplus \cE'$  provides a second pair of endomorphisms $\sx_2'
   : M' \to M'$ and $\sx_2'' : M'' \to M''$ compatible with the decomposition $M = M'
   \oplus M''$ induced by $\cE_1 = \cE_1' \oplus \cE_1''$. Again the
   Harder-Narasimhan filtration induces a decomposition $A_2 = A_2'
   \oplus A_2''$. The exact sequence defining $\cE$ as an extension of
   $\cF_1$ by $A_2 \otimes \cF(2)$ together with the direct sum
   decompositions of $\cE_1$ and $A_2$ provides elements $\zeta'$ and
   $\zeta''$ of the Ext groups:
   \[
   \zeta' \in \Ext^1_X(\cE_1',A_2' \otimes \cF(2)), \qquad    \zeta''
   \in \Ext^1_X(\cE_1'',A_2'' \otimes \cF(2)),
   \]
   and the extension providing $\cE$ is the block-diagonal sum of
   $\zeta'$ and $\zeta''$.
   This means that, denoting by $\iota : \cE_1' \to \cE_1$ the obvious
   injection, the map
   \[
   \iota ^* : \Ext^1_X(\cE_1,(A_2' \oplus A_2'') \otimes \cF(2)) \to
   \Ext^1_X(\cE_1',(A_2' \oplus A_2'') \otimes \cF(2))
   \]
   must map the class $\zeta$ of $\cE$ to $(\zeta',0)$, and similarly
   the map induced by the injection $\cE_1'' \to \cE_1$ must map
   $\zeta$ to $(0,\zeta'')$.

   In view of the description of the Ext groups we have given in
   \eqref{describeExt} and    \eqref{describeExt2}, and because
   $\iota$ corresponds to the inclusion of $M'$ into $M$, this implies
   that, up to adding $\gamma \otimes \id_W
  \circ \xi_1$ for some $\gamma : M' \to A_2' \oplus A_2''$, the map
   \[
   \xi_2 \circ \iota : M' \oplus M' \to (A_2' \oplus A_2'') \otimes W
   \]
   must have a vanishing component in $A_2'' \otimes W$.
   Write
   $\gamma$ as the transpose of $(\gamma',\gamma'')$, where $\gamma'
   \in \Hom_\bk(M',A_2')$ and  $\gamma''
   \in \Hom_\bk(M',A_2'')$.

   Denote by $\pi' : M \to A_2'$
   and $\pi'' : M \to A_2''$ the obvious projections and
   recall that, by definition, we have $\xi_2=(0,\id_M w_0 + \sx_2 w_1)$ and
   $\xi_1=(\id_M w_0 + \sx_1 w_1,0)$.
   Evaluating at
   $(w_0,w_1)=(1,0)$, we get maps $M' \oplus
   M' \to A_2''$ satisfying the following equality:
   \[
   (\gamma'' \circ \id_{M'},0) = (0,\pi'' \circ \iota).
   \]
   In particular $M' = \im(\iota) \subset \ker(\pi'')=A_2'$. Similarly
   we get
   $M'' \subset \ker(\pi')=A_2''$ and in view of the equalities $M' \oplus M'' = M =
   A_2'\oplus A_2''$, we obtain $M'=A_2'$ and $M''=A_2''$.

   Evaluating at $(w_0,w_1)=(0,1)$ we get $\pi_2'' \circ \sx_2 \circ
   \iota=0$, meaning that $\sx_2$ maps $M'=\im(\iota)$ to
   $M'=\ker(\pi'')$, namely $M'$ is stable for $\sx_2$.
   One proves similarly the same statement for $M''$.
   Therefore
   $\sx_2 : M \to M$ is the block-diagonal sum of $\sx_2' : M' \to M'$
   and $\sx_2'' : M'' \to M''$ so that $M$ is the direct sum of $M'$
   and $M''$ as $\Sigma$-modules.

   Note that, if the decomposition of $\cE$ is non-trivial (that is to
   say, if  $\cE' \ne 0 \ne \cE''$) then at least one of the decompositions of
   $A_0$, $A_1$ and $A_2$ is non-trivial, and
   therefore all of them are by what we have just seen, so the
   decomposition of $M$ as $\Sigma$-module is non-trivial too.
\end{proof}

\subsection{Proof of Theorem \ref{minimal degree}}

  The variety $X$ is a cone over a base $X^0$ which is a smooth
  irreducible variety of minimal degree. We adopt the convention of
  \S
  \ref{extensions-cones} and write $R=\bk[X]$ and $R^0=\bk[X^0]$ so $R
  \simeq R^0 \otimes
  \bk[\Lambda]$.
  We always put a $0$ superscript to sheaves on $X^0$ and modules on
  $R^0$, remove the superscript to indicate the corresponding
  object on $X$, and put a calligraphic letter for the coherent sheaf
  associated with a module denoted by that letter.

  We have three cases to check according to whether $X^0$ is a
  quadric, or a Veronese surface in $\PP^5$, or a scroll.
  These are treated in a conceptually unified way by Lemma \ref{ext-i-tensor},
  Lemma \ref{Ext-Ext} and Theorem \ref{strictly wild by extensions}:
  only the choice of the basic sheaves on $X^0$ obliges us to separate them.
  Note that, once the ACM (or Ulrich) sheaves $\cE^0$ and $\cF^0$ are stable on $X^0$,
  their lifts $\cE$ and $\cF$ are ACM (or Ulrich) and stable on
  $X$. Indeed, the ACM and Ulrich conditions are obvious as they can
  be read on the minimal graded free resolutions of $E^0$ and $E$ or
  of $F^0$ and $F$, which are unchanged on $S$ or $S^0$. Stability is
  also clear, as any destabilizing
  subsheaf, restricted to a generic linear space of dimension $n_0$,
  would destabilize $\cE^0$ or $\cF^0$.
  \smallskip

\subsubsection{Quadric cones}
\label{quadrics}

Here, $X^0$ is a quadric of corank greater than one,
in other words $\dim(\Lambda) \ge 1$.
We take $\cS^0$ to be a spinor bundle on $X^0$,
see \cite{ottaviani:spinor, buchweitz-greuel-schreyer, knorrer:ACM}.
Then, $\cE^0=\cS^0(1)$ is an Ulrich bundle and it sits in a short exact sequence
\[
0\to \cF^{0}\to\cO_{X^0}^{2\rk \cS^0}\to \cE^0\to 0,
\]
where $\cF^{0}$ is again a spinor bundle (isomorphic to $\cS^0$
if and only if the dimension of $X^0$ is odd).
Both are stable sheaves on $X^0$. We have, by Lemma \ref{Ext-Ext}:
\[
\bigoplus_{t \in \NN} \Ext^1_{X^0}(\cE^0,\cF^{0}(t)) \simeq
\Ext^1_{R^0}(E^0,F^{0}) \simeq \bk,
\]
  where $E^0 = \Gamma_*(\cE^0)$ and $F^0 = \Gamma_*(\cF^0)$.
  Therefore, defining $\cE$ and $\cF$ as sheafifications of $E=E^0
  \otimes_\bk \bk[\Lambda]$ and $F=F^0
  \otimes_\bk\bk[\Lambda]$, by Lemma \ref{ext-i-tensor}, we obtain an
  isomorphism of $R$-modules:
  \[
  \bigoplus_{t \in \NN} \Ext^1_{X}(\cE,\cF(t)) \simeq \bk[\Lambda].
  \]

  In particular the component of degree $t$, for $t \ge 2$, of this
  extension space has dimension at least $3$.
  Applying item \ref{nonstrict} of Theorem \ref{strictly wild by
    extensions} to  $\cB=\cE$ and $\cA=\cF(2)$ gives the result.

  \smallskip
\subsubsection{Cones over the Veronese surface}

If $X^0$ is the Veronese image of $\PP^2$ in $\PP^5$, we take $\cE^0$ to be the tangent bundle $T_{\PP^2}$. This is a stable Ulrich bundle with respect to $\cO_{X^0}(1)\simeq\cO_{\PP^2}(2)$. Write $W$ for the 3-dimensional space
  $\HH^0(\PP^2,\cO_{\PP^2}(1))$.
  Letting  $\Omega_{\PP^2}$ be the cotangent bundle on $\PP^2$ and
  tensoring the Euler sequence with $\Omega_{\PP^2}(t)$ we get, for $t=0$:
  \[
  \HH^1(\PP^2,\Omega_{\PP^2} \otimes T_{\PP^2}(-2)) \simeq
  \HH^2(\PP^2,\Omega_{\PP^2}(-2))  \simeq W.
  \]
  In general for $t \in \ZZ$, we obtain:
  \begin{equation*}
   \Ext^1_{X^0}(\cE^0(1),\cE^0(t)) \simeq
    \begin{cases}
      W & \mbox{if $t=0$},\\
      0 & \mbox{otherwise}.
    \end{cases}
  \end{equation*}

  Therefore we have isomorphisms of graded $R$-modules:
  \[
  \bigoplus_{t \in \NN} \Ext^1_{X}(\cE(1),\cE(t))
  \simeq   \Ext^1_{X^0}(\cE^0(1),\cE^0) \otimes
  _{\bk}\bk[\Lambda] \simeq W \otimes
  _{\bk}\bk[\Lambda].
  \]
  Applying item \ref{nonstrict} of Theorem \ref{strictly wild by
    extensions} to the pair of sheaves $\cB=\cE(1)$ and $\cA=\cE(2)$ gives the result.

  \smallskip
\subsubsection{Cones over scrolls}  Finally, assume that $X^0$ is a scroll of degree
  $d$. By Theorem \ref{funnycase} we may suppose $d \ge 4$ or $m
  \ge 3$.
  If $m=2$ (and hence $X^0 \simeq \PP^1$) we work like in Theorem \ref{funnycase} and take $\cF$
  to be the ideal sheaf of a ray of
  the cone $X$. We have $\cF^0 \simeq \cO_{\PP^1}(-1)$. The vector space
 $$W=\Ext^1_{X^0}(\cF^0(1),\cF^0) \simeq \HH^1(\PP^1,\cO_{\PP^1}(-d))$$ has
  dimension $d-1 \ge 3$. Moreover, by Lemma \ref{Ext-Ext}, Case \ref{if both Ulrich}, we get that $\Ext^1_{R^0}(F^0(1),F)_0\simeq W$. Then, we are in position to apply Theorem \ref{strictly wild by
    extensions} because, for all $q \in \NN$, the following space has dimension at least $3$:
  \[
  \Ext^1_{X}(\cF(1),\cF(q)) \simeq  \Ext^1_R(F(1),F)_{q} \simeq
  W \cdot \lambda_0^q.
  \]

  If $m \ge 3$, either $\dim(\Lambda)\ge 1$ or $\dim(X^0) \ge 2$.
  In the former case, we may assume $\dim(X^0)=1$ so again $X^0 \simeq \PP^1$ and we choose $\cF^0$
  as before. The space $W=\Ext^1_{X^0}(\cF^0(1),\cF^0)$ has positive
  dimension and we get, for $q \ge 0$:
  \[
  \Ext^1_{X}(\cF(1),\cF(q)) \simeq  \Ext^1_R(F(1),F)_{q} \simeq
  W \cdot \bk[\Lambda]_{q},
  \]
  and this space has unbounded dimension for $q \gg 0$ as
  $\dim(\Lambda) \ge 1$.
  In the latter case, we may assume $\dim(\Lambda)=0$, and choose
  $\cF^0$ to be the ideal of a fibre of the scroll, twisted by
  $\cO_{X^0}(1)$, and $\cE^0$ to be the line bundle associated with the divisor of
  $d-1$ fibres. These are both (obviously stable) Ulrich line bundles. Also,
  $W_{-1}:=\Ext^1_{X^0}(\cE^0(1),\cF^0)\simeq\HH^1(\cO_{\PP^1}(-d))\simeq
  \bk^{d-1}$, while $W_0:=\Ext^1_{X^0}(\cE^0,\cF^0)$ has dimension at least
  1, see \cite[Lemma 3.1]{miro_roig:scrolls}.
  Therefore, we get:
  \[
  \Ext^1_{X}(\cE,\cF) \simeq \Ext^1_{X}(E,F)_0 \simeq
  \Ext^1_{X^0}(E^0,F^0)_{-1} \cdot \lambda_0 \oplus
  \Ext^1_{X^0}(E^0,F^0)_0 \simeq W_{-1} \cdot \lambda_0 \oplus W_0.
  \]
  So $\dim_{\bk}\Ext^1_{X}(\cE,\cF) \ge 3$ and
   Theorem \ref{strictly wild by
    extensions} allows us to conclude.

  \begin{rmk}
    \label{char2}
    If the base field $\bk$ is algebraically closed and $\cha(\bk)=2$, all the
    statements of Theorem \ref{minimal degree} and hence of Theorem
    \ref{main-minimal} remain true, except perhaps the fact that
    quadric cones of corank one are CM-countable.

    All the proofs remain the same except when $X$ is a quadric
    hypersurface. If the quadric $X$ is smooth, then $X$ is CM-finite
    by  \cite{buchweitz-eisenbud-herzog}. If $X$ is not smooth, then
    $X$ is a cone over a smooth quadric $X^0$, the vertex of the cone being
    a linear subspace $\Lambda \subset \PP^n$ (recall that $\bk$ is algebraically
    closed). If $\dim(\Lambda)\ge 1$, again
    \cite{buchweitz-eisenbud-herzog} provides the sheaves $\cE^0$ and
    $\cF^0$ as is \S \ref{quadrics}, and these sheaves are Ulrich,
    hence semistable, as we shall see in Lemma \ref{ulrich semistable}.
    Since any destabilizing subsheaf should be Ulrich, CM-finiteness
    of $X^0$ implies that $\cE^0$ and $\cF^0$ can be chosen to be
    stable and hence simple. So $X$ is CM-wild in this case as in \S
    \ref{quadrics}.

    We do not know if a quadric cone $X$ of corank $1$ is
    CM-countable when $\cha(\bk)=2$. Indeed, our construction provides
    countably many ACM sheaves on $X$ also in this case, but Knörrer
    periodicity does not apply directly to show that these are the
    only indecomposable ACM sheaves on $X$ up to isomorphism.
\end{rmk}

\section{Varieties of almost minimal degree}
\label{section almost minimal}

From now on the field $\bk$ is algebraically closed of arbitrary
characteristic.
Let us take a further step in the proof of our main result. In view of
Theorem \ref{non-reduced}, we will assume from now on that the
subscheme $X \subset \PP^n$ is reduced, and keep the usual assumption
that $X$ is closed, non-degenerate and ACM of dimension $m\ge 1$.
In this section we pay attention to varieties of \textit{almost minimal degree}, namely the degree of
$X$ is $d=n-m+2$.
If $X$ is irreducible and normal
then $X$ is usually called a del Pezzo variety (terminology
may differ slightly in the literature).
Our goal in this section is to prove the following result.

\begin{thm} \label{resumealmost}
  Any reduced, non-degenerate ACM scheme $X\subset\PP^n$ of dimension
  $m\geq 2$ and almost minimal degree is of wild CM-type. For $m=2$,
  $X$ is strictly Ulrich wild.
\end{thm}

Let us first look briefly at reducible (and reduced) ACM subschemes of
almost minimal degree and then focus on the irreducible ones, normal or not.

\subsection{Reducible subschemes of almost minimal degree}

Let us start by assuming that $X$ is reducible, namely $X=X_1\cup X_2$
where $X_1$ and $X_2$ are closed subschemes of $\PP^n$ of degree $d_1$
and $d_2$.  We claim that not both of them have degree $d_i\geq
n_i-m+2$ at the same time, where $n_i$ denotes the dimension of the linear span
$\langle X_i\rangle$ of $X_i$ inside $\PP^n$. Otherwise,
$$
n-m+2=d=d_1+d_2\geq n_1+n_2-2m+4.
$$
On the other hand $n=n_1+n_2-l$, with $l=\dim(\langle X_1\rangle\cap\langle X_2\rangle)$. From here we would
obtain that $l\leq m-2$,
in contradiction with the fact that, since the subscheme $X$
is ACM, it must be connected in codimension one (again, see \cite{hartshorne:connectedness} for
the definition and the result) and therefore $X_1$ and $X_2$, both of
dimension $m$ should meet along a Weil divisor.

By induction on the number of components of $X$, we deduce that we can
find two irreducible components $X_1$ and $X_2$ of $X$, of dimension
$m$, such that either both of them are of minimal degree in their linear
span or one of them, say $X_1$, is of minimal degree on its linear
span and $X_2$ is of quasi-minimal degree.

In both cases, the subscheme $X_1$, being of minimal degree, is ACM in its
linear span, while $X_2$ is either ACM or the image
of a (finite) projection of a variety $\bar X_2$ of minimal
degree (see \cite[Theorem 1.2]{brodmann-schenzel:arithmetic}). Since
$\bar X_2$ supports an ACM sheaf (see \S \ref{minimal}), so does $X_2$, because
taking the direct image via a finite map preserves the ACM property.
In all these circumstances, Theorem \ref{two-pieces} applies to show that $X_1 \cup X_1$
is CM-wild, and by Lemma \ref{one-piece}, $X$ is CM-wild as well.

\subsection{Reduction to surfaces}

In view of the previous discussion, we may assume from now on that $X$
is irreducible and of almost minimal degree.
In other words, $X$ has $\Delta$-genus one. This condition amounts to
asking that the sectional $p_X$ of $X$ is one,
see \cite[page 45]{Fu90}.
Essentially, these varieties are completely classified:
see \cite{fujita:delta-genus} for the case of normal varieties, and
\cite{reid:non-normal-del-pezzo,brodmann-schenzel:arithmetic} for the
non-normal case.
For surfaces of degree $3$ and $4$ we refer to
\cite{lee-park-schenzel:cubics,lee-park-schenzel:quadrics}.
For the roots of this classification,
originated from work of Schläfli and Cayley, see for instance
\cite{abhyankar:cubic, bruce-wall:surfaces}.

Since we are assuming that $X$ is ACM, it turns out that $X$ is arithmetically
Gorenstein (AG), which is to say, $R=\bk[X]$ is a graded Gorenstein ring  (see
\cite[Remark 4.5]{brodmann-schenzel:arithmetic}). Therefore the canonical
sheaf satisfies $\omega_X \simeq \cO_X(m-1)$: indeed, by
\cite[Corollary 4.1.5]{migliore:liaison}, $\omega_X$ must be of the
form $\cO_X(t)$ for some $t \in \ZZ$; on the other hand restricting to a generic
one-dimensional linear section $Y$ gives
 $\omega_Y\simeq \cO_Y(t+m-1)$ by adjunction because $p_Y=1$. Therefore
 $\chi(\omega_Y)=0$ so $t=1-m$.

Let us quickly rule out the case $m=1$. We have that $X$ is of tame CM-type if it
smooth (namely $X$ is an elliptic curve), by classical work of
Atiyah, see \cite{atiyah:elliptic}.
If $X$ is singular, we know by \cite{drozd-greuel} that $X$ is CM-tame
if $X$ a cycle of rational curves with ordinary double points, and that
$X$ is CM-wild otherwise.

The goal of this section is to deal with higher dimensions. Namely we
want to prove that, if $X$ is ACM of almost minimal degree, then $X$
is CM-wild as soon as $m \ge 2$.

The idea is to use our reduction to linear sections. Let $Y\subset \PP^d$ be a linear
section of $X$ with $\dim(Y)=2$. By Bertini's theorem, we may assume that $Y$ is also an
 irreducible ACM subscheme
  of almost minimal degree, with $\omega_Y \simeq \cO_Y(-1)$. Note that the codimension $c$ of $Y$ in
  $X$ is $m-2$ so $\omega_Y(m-c-1) \simeq \cO_Y$. Therefore,
 Theorem
\ref{embedding} applies to prove that $X$ is CM-wild, as soon as we show that $Y$
supports two simple Ulrich sheaves $\cA$ and $\cB$ such that:
\[
\Hom_Y(\cA,\cB)=0=\Hom_Y(\cB,\cA), \quad \mbox{and} \quad \dim_\bk
\Ext^1_Y(\cB,\cA) \ge 3.
\]
Finding $\cA$ and $\cB$ as above will be our task.  It is natural to
look for $\cA$ and $\cB$ among sheaves of low rank.
However, in general it will not be possible to obtain sheaves of rank one. Indeed, Ulrich
sheaves of rank one may not exist for certain del Pezzo surfaces,
for example cubic surfaces with an $E_6$ singularity (see \cite[Theorem 9.3.6]{dolgachev:classical-AG}).

Therefore we move forward to construct rank two sheaves using the  Hartshorne-Serre
correspondence. Of course this idea is not new, as for instance it is
widely used in \cite{casanellas-hartshorne:ACM-cubic} precisely to
construct families of Ulrich bundles on smooth cubic surfaces, a
special case of surfaces of almost minimal degree.
The construction can be performed in quite a general setup; for
instance, for cubic surfaces one knows, even if $\bk$ is not
algebraically closed, the degree of the field extension needed to
construct $\cE_Z$, see \cite{tanturri:pfaffian}.
However, in our setting we have to
be a bit more careful since the surfaces under consideration may be badly singular.

\subsection{Surface cones}

Let us quickly rule out the case of cones, namely assume that
$Y\subset \PP^{d}$ is a cone over an integral curve $Y^0 \subset
\PP^{d-1}$ with trivial canonical sheaf, the vertex of the cone being
a single point.
In this case, using \cite[Proposition 3.5]{altman-kleiman:compactifying},
we may choose non-isomorphic Ulrich line bundles $\cF^0$ and $\cE^0$ on the curve
$Y^0$, put $\cM = \cE^0 \otimes (\cF^0)^\vee$ and observe:
\[
\Ext^1_{Y^0}(\cE^0,\cF^0(-1))^* \simeq \HH^0(Y^0,\cM(1)) \simeq \bk^{d}.
\]
because $\cM(1)$ is a line bundle of degree $d$, and clearly $d \ge
3$.
We also have $\Ext^1_{Y^0}(\cE^0,\cF^0)^* \simeq \HH^0(Y^0,\cM)=0$
because $\cE^0$ and $\cF^0$ are not isomorphic. By the same reason we have
$\Hom_{Y^0}(\cE^0,\cF^0) = \Hom_{Y^0}(\cF^0,\cE^0) = 0$.

We use the approach of \S \ref{extensions-cones}.
The coordinate ring $R=\bk[X]$ takes the form $R^0 \otimes_\bk
\bk[\lambda]$, with $R^0=\bk[X^0]$ and the modules $E^0 =
\Gamma_*(\cE^0)$ and $F^0 = \Gamma_*(\cF^0)$ give rise, by tensoring
with $\bk[\lambda]$, to Ulrich $R$-modules $F$ and $E$ (as their
$S$-resolutions are still linear) and thus to Ulrich
sheaves $\cE$ and $\cF$ over $X$. By Lemma \ref{ext-i-tensor} and
Lemma \ref{Ext-Ext}, we have:
\[
\Ext^1_Y(\cE,\cF) \simeq \Ext^1_{Y^0}(\cE^0,\cF^0(-1)) \cdot \lambda,
\]
so this space has dimension at least $3$. The same argument shows
$\Hom_{Y}(\cE,\cF) = \Hom_{Y}(\cF,\cE) = 0$. Then $\cA=\cE$ and
$\cB=\cF$ are the desired  sheaves to apply  Theorem
\ref{embedding}.

\subsection{Hartshorne-Serre correspondence}

At this point we may assume that the surface $Y \subset \PP^d$ is not a
cone.  Let us consider a set $Z\subset Y \subset \PP^d$ of $d+2$
distinct
points points in general linear position and disjoint from $\Sing(Y)$. We have:
\[
\Ext^1_Y(\cI_{Z|Y}(2),\cO_Y) \simeq \Ext^1_Y(\cO_Y,\cI_{Z|Y}(1))^*
\simeq\HH^1(Y,\cI_{Z|Y}(1))^* \simeq \bk,
\]
 where we used that $\omega_Y\simeq\cO_Y(-1)$ and that $Z$ spans the
 whole $\PP^d$. A non-zero element
 $\lambda \in \HH^1(Y,\cI_{Z|Y}(1))^*$ provides a coherent sheaf
 $\cF_Z$ of rank $2$ that fits into the short exact sequence:
\begin{equation}\label{Serre}
0\to\cO_Y\to\cF_Z\to\cI_{Z|Y}(2)\to 0.
\end{equation}
Given a set of points $Z$ of $d+2$ points of $Y$, write $[Z]$ for the
corresponding element of the Hilbert scheme $\Hilb^{d+2}(Y)$ of
subschemes of length $d+2$ of $Y$.
In the next lines, stability is always with respect to $\cO_Y(1)$.

\begin{lem}
The sheaf $\cF_Z$ is Ulrich and locally free of rank $2$.
\end{lem}

\begin{proof}
Let us prove that $\cF_Z$ is locally free. We need only prove this
around any point $z$ of $Z$, as $\cI_{Z|Y}$ is already free of rank
$1$ away from $Z$.
First of all, taking the dual of the short exact sequence
\[
0\to\cI_{Z|Y}(2)\to\cO_Y(2)\to\cO_Z(2)\to 0,
\]
we deduce, since $Z$ is smooth and zero-dimensional, that
\[
\cExt^1_{\cO_Y}(\cI_{Z|Y}(2),\cO_Y)\simeq\cExt^2_{\cO_Y}(\cO_Z,\cO_Y)\simeq\omega_Z.
\]

Next, note that  $\cHom_{\cO_Y}(\cI_{Z|Y}(2),\cO_Y)) \simeq
\cO_Y(-2)$. In view of the vanishing
$\HH^1(Y,\cO_Y(-2))=0$, by
 the local-to-global spectral
sequence we get an exact sequence:
\[
0 \to \Ext^1_Y(\cI_{Z|Y}(2),\cO_Y) \to \HH^0(Y,\cExt^1_{\cO_Y}(\cI_{Z|Y}(2),\cO_Y))
\to \HH^2(Y,\cO_Y(-2)) \to 0.
\]
Using Serre duality and the above isomorphisms we rewrite this as
\[
0 \to \Ext^1_Y(\cI_{Z|Y}(2),\cO_Y) \to \HH^0(Y,\omega_Z)^*
\to \HH^0(Y,\cO_Y(1))^* \to 0.
\]
We may choose coordinates so that $Z$ is the union of $d+2$ points of
a projective coordinate system, so that $\lambda$ is the vector
$(1,\ldots,1,-1)$ in $\HH^0(Y,\omega_Z)^*$, which shows that $\lambda$
is non-zero at any point of $Z$.

Therefore, $\lambda$ corresponds to a global section:
\[
 \cO_Y \to \cExt^1_{\cO_Y}(\cI_{Z|Y}(2),\cO_Y) \simeq \omega_Z,
\]
which is non-zero at any point $z$ of $Z$. Since $z$ is locally defined by
two equations, the sheaf $\cExt^1_{\cO_Y}(\cI_{z|Y}(2),\cO_Y)$ is
one-dimensional at $z$, generated by the extension given by the Koszul complex of these
equations; so the middle term of such extension is $\cF_Z$. Hence
$\cF_Z$ is locally free around any point $z$ of $Z$.
\medskip

The fact that the sheaf $\cF_Z$ is Ulrich follows from \cite[Proposition 2.1]{eisenbud-schreyer-weyman}.
Indeed, we need to prove that
$\HH^i(Y,\cF_Z(-1))=\HH^i(Y,\cF_Z(-2))=0$ for all $i$. Since $\cF_Z$ is
locally free of rank $2$ and clearly $\wedge^2 \cF_Z \simeq \cO_Y(2)$,
we have:
\[
\HH^{2-i}(Y,\cF_Z(-2))^* \simeq \HH^i(Y,\cF_Z^\vee(2) \otimes \omega_Y) \simeq \HH^i(Y,\cF_Z(-1)),
\]
so it will be enough to prove one set of vanishing conditions. This
amounts to checking that the map $\HH^1(Y,\cI_{Z|Y}(1)) \to
\HH^2(Y,\cO_Y(-1))$ is an isomorphism. We observe that this map is Serre-dual of the map $\HH^0(Y,\cO_Y) \to \Ext^1_Y(\cI_{Z|Y}(2),\cO_Y)$ that sends the
identity to $\lambda$, and therefore it is an isomorphism.

Otherwise, one may deduce that $\cF_Z$ is Ulrich by the form of the minimal
graded free resolution of the ideal of $Z$, which can be extracted from
 \cite{miro-roig-pons-llopis:del-pezzo}.
\end{proof}

\begin{lem} \label{ulrich semistable}
  Let $\cF$ be an Ulrich sheaf on an $m$-dimensional closed subscheme
  $X \subset \PP^n$. Then $\cF$ is semistable and any destabilizing
  subsheaf of $\cF$ is Ulrich.
\end{lem}

\begin{proof}
  This follows again from \cite[Proposition
  2.1]{eisenbud-schreyer-weyman}.
  Indeed, first note that, since $\cF$ is Ulrich, it is also locally
  Cohen-Macaulay and therefore pure.
  Next, choosing a finite linear projection
  $\pi : X \to \PP^m$, we have $\pi_*(\cF) \simeq \cO_{\PP^m}^u$ for
  some integer $u$. Put $\chi_m=\pp(\cO_{\PP^m})$ and $d=\deg(X)$.

  Suppose that  $\cF'$ is a proper subsheaf of $\cF$   with $\pp(\cF') \succ \pp(\cF)$.
  Since $\pi$ is finite, we have:
  \[
  P(\cF',t)\frac{\rk(\cF)}{u \rk(\cF')} \succ \chi_m(t),
  \]
  and, because $\cO_{\PP^m}^u$ is semistable,
  $P(\cF',t)/\rk(\pi_*(\cF')) \preceq \chi_m(t)$.
  But this contradicts the equality
  \[
  d = \frac{\rk(\pi_*(\cF'))}{\rk(\cF')}=\frac{u}{\rk(\cF)}.
  \]
  This shows that $\cF$ is semistable. Moreover, if $\pp(\cF') =
  \pp(\cF)$, then $\pp(\pi_*(\cF'))=\chi_m(t)$ so $\pi_*(\cF') \simeq
  \cO_{\PP^m}^{u'}$ for some integer $u'$ because $\cO_{\PP^m}^u$ is polystable, which implies that $\cF'$ is Ulrich.
\end{proof}

\begin{prop}
  The set $Z$ can be chosen so that $\cF_Z$ is stable.
\end{prop}

\begin{proof}
We know that $\cF_Z$ is Ulrich so either $\cF_Z$ is stable or there
exist $\cA$ and $\cB$ Ulrich sheaves of rank $1$ such that $\cF_Z$
fits into:
\[
0\to\cA\to\cF_Z\to\cB\to 0.
\]

Note that $\cA$ is reflexive as $\cB$ is torsion-free and $\cF_Z$ is
locally free. Therefore, a global section of $\cA$ vanishes along a
subscheme of $Y$ which is Cohen-Macaulay of dimension one (see \cite[Proposition~2.8 and~2.9]{hartshorne-k-theory}).

By construction, the global section $s\in\HH^0(Y,\cF_Z)$ associated
with $Z$ vanishes precisely on $Z$. Therefore $s$ cannot lie in
$\HH^0(Y,\cA) \subset \HH^0(Y,\cF_Z)$, as $s$ would then vanish
in codimension $1$. Hence we can construct the following commutative diagram:
\[
\xymatrix@-2ex{
   &   &  \cO_Y \ar@{=}[r] \ar[d]^s & \cO_Y \ar[d] & \\
  0 \ar[r] & \cA \ar[r] \ar@{=}[d] &  \cF_Z \ar[r] \ar[d] & \cB \ar[r]\ar[d] & 0\\
  0 \ar[r] & \cA \ar[r] & \cI_{Z|Y}(2) \ar[r]   & \cT \ar[r] & 0
}
\]
where $\cT$ is a torsion sheaf defined by the diagram.
This tells us that:
\[\
\HH^0(Y,\cI_{Z| Y} \ts \cA^\vee(2)) \ne 0,
\]
namely $Z$ lies on a divisor $D$ from the linear system
$|\cA^{\vee}(2)|$. Our goal is to prove that $Z$ can be chosen so that
it lies on no such divisor. Notice that, by \cite[Lemma 2.4]{casanellas-hartshorne-geiss-schreyer}, $\cA^\vee(2)$ is also an Ulrich sheaf of rank one.

\medskip

Assume first that $Y$ is normal. For each fixed Ulrich sheaf $\cL$ of
rank one, we know that $\dim_{\bk} \HH^0(Y,\cL)=d$, and each
non-zero global section of $\cL$ vanishes along a Weil divisor $D
\subset Y$. We view isomorphism classes of such sheaves as
elements of the divisor class group of $Y$, which  is
identified with the group $\APic(Y)$ of generalized divisors on $Y$,
see \cite[\S 2]{hartshorne:generalized-divisors}.

Each linear system $|\cL|$ has dimension $d-1$. For each $D$ in $|\cL|$, taking all
non-degenerate smooth subschemes $Z \subset Y$ lying in $D$ we obtain
a non-empty open subset of the
Hilbert scheme $\Hilb_{d+2}(D)$ which is $(d+2)$-dimensional. The union of these such $Z$ for all
choices of $D$ in $|\cL|$ forms a subscheme of $\Hilb_{d+2}(Y)$ which
is of dimension at most $(d+2)+\dim |\cL|=2d+1$. But the main
component of $\Hilb_{d+2}(Y)$ (that is, the component containing smooth
subschemes) has dimension $2d+4$, so we may choose
$Z$ not lying in any $D \in |\cL|$.
Finally, since the divisor class group of a normal del Pezzo surface
is discrete, we may choose $Z$ away from the union, over all divisor
classes arising from Ulrich sheaves $\cL$, of the subschemes of
$\Hilb_{d+2}(Y)$ associated with $D$ lying in $|\cL|$.
So $\cF_Z$ is stable if $Z$ is general enough.

\medskip
Now let us assume that $Y$ is not normal.
We know by \cite{reid:non-normal-del-pezzo} that $Y$ is normalized by
a surface $\bar Y \subset \PP^{d+1}$ of minimal degree $d$, the
normalization map $\bar Y \to
Y$ being induced by a projection $\PP^{d+1} \to \PP^d$.
The normalization is an isomorphism away from a conic in $\bar Y$ which is mapped onto the
singular locus of $Y$, which in turn is a line $L$. Moreover, the
surface $\bar Y$ is smooth as otherwise, being of minimal degree, it would have to be a cone, but
then $Y$ would be a cone too, which we excluded.

Given an Ulrich sheaf $\cL$ of
rank one, again we have $\dim_{\bk} \HH^0(Y,\cL)=d$, and we choose a
non-zero global section of $\cL$. This vanishes along a Weil divisor $D
\subset Y$ of degree $d$, which contains a structure of multiplicity
$e \le d$ over $L$. Removing this structure from $D$ we obtain an effective
generalized divisor $D_0$, whose class lies in $\APic(Y)$, see
\cite[Proposition 2.12]{hartshorne:generalized-divisors}.
Since $Z$ is disjoint from $L$, it will be enough to prove that we may
choose $Z$ away from any divisor $D_0$ of degree $d-e$, and obviously
it suffices to show that this holds for $D_0$ of degree $d$.

To do this we use the explicit description of $\APic(\bar Y)$ given by
\cite[Theorem 4.1 and Proposition 4.2]{hartshorne-polini:divisor}.
Indeed, $\bar Y$ is either a Veronese surface in $\PP^5$ (and thus $d=4$) or a rational normal scroll
of degree $d \ge 3$, and $D_0$ is the image of an effective divisor
$D_0'$ of degree $d$ in $\bar Y$. Also, taking a hyperplane section $C
\simeq \PP^1$ of $\bar Y$, we get:
\[
0 \to \cO_{\bar Y}(D'-H) \to \cO_{\bar Y}(D') \to
\cO_C(D') \to 0,
\]
so $\dim_\bk \HH^0(\bar
Y,\cO_{\bar Y}(D')) \le d+2$, as $\deg(D'-H)\le 0$  and
$\cO_C(D') \simeq \cO_{\PP^1}(\deg(D'))$, with $\deg(D')
\le d$.

Therefore, since there are finitely many effective divisor classes of
degree at most $d$ in $\bar Y$, if we choose $Z' \subset \bar Y$ to be a
non-degenerate set of $d+2$ distinct points lying away from all
divisors $D'$ in those classes, the image $Z$ of $Z'$ in $Y$ will
be contained in no generalized divisor $D_0$ of degree at most $d$. We
conclude that $\cF_Z$ is stable.
\end{proof}

\begin{lem} 
We may choose $Z$ and $Z'$ sets of $d+2$ points of $Y$ such that $\cE=\cF_Z$,
$\cF=\cF_{Z'}$ are non-isomorphic stable Ulrich bundles of rank $2$ on
$Y$. In this case:
\begin{align*}
&\Hom_Y(\cE,\cF) = \Hom_Y(\cF,\cE)= 0, \\
&\dim_{\bk} \Ext^1_Y(\cE,\cF)=4.
\end{align*}

\end{lem}

\begin{proof}
We have proved so far that, for $Z$ general enough in the main
component of $\Hilb_{d+2}(Y)$, the sheaf $\cF_Z$ is a stable locally
free Ulrich sheaf of rank $2$. Given such $Z$, choosing a non-zero
global section of $\cF_Z$ gives a map from an open dense subset of
$\PP(\HH^0(Y,\cF_Z)) \simeq \PP^{2d-1}$ to the main component of
$\Hilb_{d+2}(Y)$ associating with the section its vanishing locus.
This map cannot be surjective by dimension reasons, so we can take $Z'$
general enough, lying away from the image of this map and such that $\cF_{Z'}$
is also a stable locally free Ulrich sheaf of rank $2$. Because $Z'$
is not the vanishing locus of a global section of $\cF_Z$, we have
that $\cF_Z$ and $\cF_{Z'}$ are not isomorphic.

The first two statements concerning morphisms are clear since  $\cE$
and $\cF$ are stable with the same slope and not isomorphic. For the
last statement, since $\cE$ is locally free, we have
$\Ext^1_Y(\cE,\cF) \simeq \HH^1(Y,\cE^{\vee}\otimes\cF)$.
Tensoring the short exact sequence \ref{Serre} by $\cE^{\vee}$ and
considering the associated long exact sequence of global sections,
since $\HH^1(Y,\cE^{\vee})=\HH^2(Y,\cE^{\vee})=0$, we
get an isomorphism
\begin{equation}
  \label{isoext}
\Ext^1_Y(\cE,\cF)  \simeq \HH^1(Y,\cE^{\vee}\otimes\cI_{Z|Y}(2)).
\end{equation}

On the other hand, the exact sequence defining $Z\subset Y$
twisted by $\cO_Y(2)$ reads:
\begin{equation}
  \label{OZ}
0\to\cI_{Z|Y}(2)\to\cO_Y(2)\to\cO_Z(2)\to 0.
\end{equation}
Taking into account that $\cE^{\vee}\simeq\cE(-2)$ we obtain
$\HH^0(Y,\cE^{\vee}\otimes\cI_{Z|Y}(2))=0$.
Then, tensoring \eqref{OZ} by $\cE^{\vee}$, taking global sections and
combining with \eqref{isoext} we get:
$$
0 \to
\HH^0(Y,\cE^{\vee}(2))\to\HH^0(Y,\cE^{\vee}\otimes\cO_Z(2))\to \Ext^1_Y(\cE,\cF)
\to 0.
$$
Now we know that $\cE^{\vee}(2) \simeq \cE$ has $2d$ independent
global sections, as $\cE$ is Ulrich. On the other hand, since
$Z$ has length $d+2$, so that $\cE^{\vee}\otimes\cO_Z(2)$ is just a
vector space of rank $2$ concentrated at $d+2$ points, hence
$\dim_{\bk}\HH^0(Y,\cE^{\vee}\otimes\cO_Z(2))=2d+4$. We conclude that $\dim_{\bk} \Ext^1_Y(\cE,\cF)=4$.
\end{proof}

Theorem \ref{embedding} combined with  the results of this section
yields the proof of Theorem \ref{resumealmost}.

\section{Varieties of higher degree}
\label{higher degree}

In this final section we prove our main theorem for subschemes of degree higher than almost
minimal, that is, in the range $\Delta(X) > 1$.
Let again $X \subset \PP^n$ be a reduced ACM subscheme of dimension $m
\ge 1$ and degree $d > n-m+2$, or in other words $\Delta(X)>1$.
Thus $p \ge 2$ (see \cite[(6.4.5)]{Fu90}).

We take a linear section $Y$ of dimension $1$, which we may assume to
be reduced, so $Y$ is an ACM curve of arithmetic genus $p \ge 2$.
We note that the proof of \cite[Proposition 3.5]{altman-kleiman:compactifying}
applies to $Y$ as it only uses the fact that the projective curve $Y$
is reduced and connected.
Then, we may find a line bundle $\cL_1$ on $Y$ satisfying:
\[
\HH^0(Y,\cL_1)=\HH^1(Y,\cL_1)=0.
\]
Therefore, $\cL_1(1)$ is an Ulrich line bundle by \cite[Theorem 4.3]{eisenbud-schreyer-weyman}
Clearly, $\Hom_Y(\cL_1,\cL_1) \simeq \HH^0(Y,\cO_Y) \simeq \bk$
because $\cL_1$ is invertible, so $\cL_1$ is simple.
Moreover the space $\Ext^1_Y(\cL_1,\cL_1) \simeq \HH^1(Y,\cO_Y)$ has
dimension $p \ge 2$ and $\Ext^2_Y(\cL_1,\cL_1)=0$.
So we may take general flat deformations of $\cL_1$ to get
sheaves $\cL_2$, $\cL_3$, $\cL_4$, not isomorphic to one another nor to $\cL_1$, which will also be
invertible (hence simple) and satisfy
$\HH^0(Y,\cL_i)=\HH^1(Y,\cL_i)=0$ for all $i$ by semicontinuity.

We claim that we may assume $\Hom_Y(\cL_i,\cL_j)=0$ for $i \neq
j$.
Indeed, first note that the degree of the line bundle $\cL_1$ on each
irreducible component of $Y$ is
constant along small deformations so we may assume that all the sheaves $\cL_i$ have the
same degree along each component.
Put $\cM = \cL_i^\vee \otimes \cL_j$, take a non-zero morphism $\cL_i
\to \cL_j$ and rewrite it as a nonzero global section $\varphi :
\cO_Y \to \cM$. The restriction of $\cM$ to any irreducible component
of $Y$ is a line bundle of degree 0, so $\varphi$ is an
isomorphism as soon as its restriction to all such components is
non-zero. Set $Y'$ for the union in $Y$ of the irreducible components of $Y$
where $\varphi$ is non-zero (and hence an isomorphism) and put $Y''$
for the closure in $Y$ of $Y\setminus Y'$, $\cM''= \cM|_{Y''}$.
By assumption $Y' \ne \emptyset \ne Y''$. Then we have the commutative
exact diagram:
\[
\xymatrix@-2ex{
0 \ar[r] & \cI_{Y' | Y} \ar[r] \ar^-{\varphi'}[d] & \cO_Y \ar^-\varphi[d] \ar[r] & \cO_{Y''} \ar[r] \ar^-{\varphi''}[d]& 0 \\
0 \ar[r] & \cI_{Y' | Y} \otimes \cM \ar[r] & \cM \ar[r] & \cM'' \ar[r] & 0
}
\]
Here, $\varphi'' = \varphi|_{\cO_Y''}$ is zero, while $\varphi'$ is
the restriction of $\varphi$ to $Y'$, tensored with the identity over
$\cI_{Y' | Y}$, and therefore is an isomorphism. Hence the snake lemma gives a splitting $\cO_{Y''} \to
\cO_Y$ of the surjection $\cO_Y \to \cO_{Y''}$, so $Y$ cannot be
connected unless $Y'$ or $Y''$ are empty, a contradiction.

As a consequence, we get that the space $\Ext^1(\cL_i,\cL_j) \simeq \HH^1(Y,\cL_i^\vee \otimes
\cL_j)$ has dimension $p-1$ for $i \neq j$.
Now we may choose
$\cA$ and $\cB$ as two sheaves given by non-trivial extensions:
\begin{align*}
&  0 \to \cL_1(1) \to \cA \to \cL_2(1) \to 0, \\
&  0 \to \cL_3(1) \to \cB \to \cL_4(1) \to 0.
\end{align*}
It is clear that $\cA$ and $\cB$ are locally free Ulrich sheaves of
rank $2$. Also, the sheaves $\cA$ and $\cB$ are simple and satisfy (see for instance \cite[Proposition 5.1.3]{pons_llopis-tonini}:
\begin{align*}
&  \Hom_Y(\cA,\cB)=\Hom_Y(\cB,\cA)=0. \\
&  \chi(\cA,\cB)=\chi(\cB,\cA)=4(1-p).
\end{align*}
We obtain the following:
\[
\dim_{\bk} \Ext^1_Y(\cA,\cB)=\dim_{\bk} \Ext^1_Y(\cB,\cA)=4(p-1) \ge 4.
\]
Note that the non-vanishing condition of Theorem \ref{stable FF}
reduces to $\HH^0(Y,\omega_Y) \ne 0$, which is true because
$p \ge 2$.
Now Theorem \ref{embedding} implies that $X$ is of wild CM representation type.


\end{document}